\documentclass[titlepage,12pt]{article} 
\usepackage{hyperref}
\usepackage[usenames,dvipsnames]{pstricks} 
\usepackage{pst-plot} 
\usepackage{amssymb,amsthm,amsmath} 
\usepackage[a4paper]{geometry}
\usepackage{datetime2}
\usepackage[utf8]{inputenc}
\usepackage[italian,english]{babel}
\usepackage{mathrsfs}
\usepackage{bbm}

\date{}

\newcommand{\re}{\mathbb{R}}

\newcommand{\hatu}{\widehat{u}}

\newcommand{\DPM}{\operatorname{DPM}}
\newcommand{\TV}{TV}

\newtheorem{thm}{Theorem}[section]
\newtheorem{thmbibl}{Theorem}

\newtheorem{rmk}[thm]{Remark}
\newtheorem{prop}[thm]{Proposition}
\newtheorem{defn}[thm]{Definition}

\newtheorem{lemma}[thm]{Lemma}

\title{
Monotonicity properties of limits of solutions to the semi-discrete scheme for the Perona-Malik equation
}

\author{
Massimo Gobbino\vspace{1ex}\\ 
{\normalsize Università degli Studi di Pisa} \\
{\normalsize Dipartimento di Matematica}\\ 
{\normalsize PISA (Italy)}\\  
{\normalsize e-mail: \texttt{massimo.gobbino@unipi.it}}
\and
Nicola Picenni\vspace{1ex}\\ 
{\normalsize Scuola Normale Superiore} \\
{\normalsize PISA (Italy)}\\
{\normalsize e-mail: \texttt{nicola.picenni@sns.it}}
}

\begin{document}
\maketitle

\begin{abstract}

We consider generalized solutions of the Perona-Malik equation in dimension one, defined as all possible limits of solutions to the semi-discrete approximation in which derivatives with respect to the space variable are replaced by difference quotients. 

Our first result is a pathological example in which the initial data converge strictly as bounded variation functions, but strict convergence is not preserved for all positive times, and in particular many basic quantities, such as the supremum or the total variation, do not pass to the limit. Nevertheless, in our second result we show that all our generalized solutions satisfy some of the properties of classical smooth solutions, namely the maximum principle and the monotonicity of the total variation.

The verification of the counterexample relies on a comparison result with suitable sub/supersolutions. The monotonicity results are proved for a more general class of evolution curves, that we call $uv$-evolutions.

\vspace{6ex}

\noindent{\bf Mathematics Subject Classification 2020 (MSC2020):} 
35D30, 35B50, 35K59, 47J06.

\vspace{6ex}

%35D30: Weak solutions to PDEs
%35B50: Maximum principle in the context of PDEs
%35K59: Quasilinear parabolic equations
%47J06 Nonlinear ill-posed problems
%35B45: A priori estimate in the context of PDEs
%35K55: Nonlinear parabolic equations
%35B51: Comparison principle in the context of PDEs
%49J45: Methods involving semicontinuity and convergence; relaxation
%39A12: Discrete version of topics in analysis
%35Q68: PDEs in connection with computer science
%68U10 Computing methodologies for image processing
%65M12 Stability and convergence of numerical methods for initial value and initial-boundary value problems involving PDEs
%49M25 Discrete approximations in optimal control
%94A08 Image processing (compression, reconstruction, etc.) in information and communication theory

%49J45 Methods involving semicontinuity and convergence; relaxation
%35B27 Homogenization in context of PDEs; PDEs in media with periodic structure
%35A15 Variational methods applied to PDEs
%35B25 Singular perturbations in context of PDEs
%35J20 Variational methods for second-order elliptic equations
%35J35 Variational methods for higher-order elliptic equations
%47J30 Variational methods involving nonlinear operators
%49Q20 Variational problems in a geometric measure-theoretic setting		
%34E10 Perturbations, asymptotics of solutions to ordinary differential equations


\end{abstract}

%%%%%%%%%%%%%%%%%%%%%
%                   %
%   Inizio lavoro   %
%                   %
%%%%%%%%%%%%%%%%%%%%%
 
\section{Introduction}  

In this paper we consider the evolution equation
\begin{equation}
u_{t}=\left(\frac{u_{x}}{1+u_{x}^{2}}\right)_{x}=\frac{1-u_{x}^{2}}{(1+u_{x}^{2})^{2}}\,u_{xx}
\label{defn:PM-eqn}
\end{equation}
in an interval $(a,b)\subseteq\re$, with boundary conditions (usually of homogeneous Neumann type), and an initial datum at $t=0$.

Equation (\ref{defn:PM-eqn}) is the one-dimensional version of the celebrated equation introduced by P.~Perona and J.~Malik in~\cite{PeronaMalik} as a tool for image denoising. The main feature is that it is a \emph{forward parabolic} equation in the subcritical regime where $|u_{x}|<1$, and a \emph{backward parabolic} equation in the supercritical regime where $|u_{x}|>1$. This is bad news, because backward parabolic evolution equations are notoriously ill-posed.

From the variational point of view, this bad news is related to the fact that (\ref{defn:PM-eqn}) is, at least formally, the gradient flow of the so-called Perona-Malik functional
\begin{equation}
\operatorname{PM}(u):=\frac{1}{2}\int_{a}^{b}\log\left(1+u_{x}(x)^{2}\right)\,dx,
\nonumber
\end{equation}
whose Lagrangian 
\begin{equation}
\varphi(\sigma):=\frac{1}{2}\log(1+\sigma^{2})
\qquad
\forall\sigma\in\re
\label{defn:varphi}
\end{equation}
is convex only in the region where $|u_{x}|<1$ and, even worse, it has a convexification that is identically zero. A similar qualitative behavior happens whenever the Lagrangian has a convex-concave behavior and sublinear growth, as for example when
\begin{equation}
\varphi(\sigma)=\arctan(\sigma^{2})
\qquad\qquad\text{or}\qquad\qquad
\varphi(\sigma)=\sqrt[4]{1+\sigma^{2}},
\label{defn:varphi-more}
\end{equation}
in which case equation (\ref{defn:PM-eqn}) takes the more general form
\begin{equation}
u_{t}=(\varphi'(u_{x}))_{x}=\varphi''(u_{x})u_{xx}.
\label{defn:PM-eqn-phi}
\end{equation}

Despite the analytical ill-posedness, numerical simulations both for (\ref{defn:PM-eqn}) and for the original two dimensional model exhibit more stability than expected. A satisfactory theory that explains this unexpected stability, usually called ``the Perona-Malik paradox'' after~\cite{Kichenassami}, is still out of reach. By ``satisfactory theory'' we mean a notion of solution for (\ref{defn:PM-eqn}) that exists for a reasonable class of initial data, depends in a reasonable way on the initial condition, and to which solutions of approximating models converge.

\paragraph{\textmd{\textit{The semi-discrete scheme and generalized solutions}}}

Several variants have been introduced in the last 30 years in order to mitigate the ill-posed nature of (\ref{defn:PM-eqn}). These variants involve either convolutions~\cite{1992-SIAM-Lions}, or a time-delay~\cite{2007-Amann}, or fractional derivatives~\cite{2009-JDE-Guidotti}, or addition of a dissipative term (see~\cite{2020-JDE-BerSmaTes} and the references quoted therein), or higher order singular perturbations~\cite{1996-Duke-DeGiorgi,2008-TAMS-BF,2019-SIAM-BerGiaTes,GP:fastPM-CdV}, or semi-discrete schemes where space derivatives are replaced by finite differences~\cite{2001-CPAM-Esedoglu,2006-SIAM-BelNovPao,GG:grad-est,2008-JDE-BNPT,2011-M3AS-BelNovPao,2011-SIAM-SlowSD}.  All these models involve some small parameter, such as the width of the time-delay, the coefficient of the higher order perturbation, or the size of the grid in the discrete approximation. All these models behave quite well for any fixed positive value of the parameter. Problems arise when one looks for estimates that are uniform with respect to the parameter, and therefore can be useful in order to pass to the limit when the parameter vanishes.

As far as we know, estimates of this type are known only for the semi-discrete scheme in dimension one (see~\cite{GG:grad-est,2008-JDE-BNPT}), and for this reason in this paper we focus on this approximation. In this case it is known that, for any fixed value of the parameter, both the $L^{\infty}$ norm and the total variation of the solution are nonincreasing functions of time (see Theorem~\ref{thmbibl:existence}). Therefore, if they are bounded for $t=0$ independently of the discretization parameter, then they remain uniformly bounded for all positive times, and this is enough to show that a limit exists, at least up to subsequences, and it is a bounded variation function for every $t\geq 0$ (see Theorem~\ref{thmbibl:compactness}). At this point one can consider all possible limits of semi-discrete approximations as generalized solutions to (\ref{defn:PM-eqn}) (see Definition~\ref{defn:gen-sol}).

A characterization of all possible limits is still out of reach, apart from some partial results (see Theorem~\ref{thmbibl:v}), because it is not clear how to pass to the limit in the quasi-linear term. In particular, it was not known whether some good properties of the approximating solutions, such as the maximum principle and the monotonicity of the total variation, remain valid for all possible limits. We observe that these properties are extremely reasonable for what is expected to be a denoising tool. In this paper we investigate these questions.

\paragraph{\textmd{\textit{First main result -- A counterexample}}}

We already observed that, for every fixed value of the grid size, both the $L^{\infty}$ norm and the total variation of solutions are nonincreasing functions of time. Therefore, both the maximum principle and the monotonicity of the total variation for generalized solutions would be trivially true if we knew that ``the total variation of the limit is the limit of the total variations'' or ``the maximum of the limit is the limit of the maxima''. 

Unfortunately, the $L^{2}$ convergence provided by the compactness result is not enough to pass such quantities to the limit for all times, even if they pass to the limit at the initial time.  Indeed, in the first result of this paper (Theorem~\ref{main:no-strict}) we provide an explicit example of solutions to the semi-discrete scheme whose maximum and total variation \emph{pass to the limit at the initial time, but do not pass to the limit for every subsequent time}.

In a nutshell, what happens in that example can be described as follows. Assume that we consider equation (\ref{defn:PM-eqn}) in the interval $(0,\pi)$, with initial datum $u_{0}(x)=c_{0}\sin x$ for some $c_{0}\in(0,1)$, and (for simplicity) Dirichlet boundary conditions. In this case the problem has a unique classical solution $u(t,x)$, which always satisfies $0\leq u_{x}(t,x)\leq c_{0}$, and does not even realize that a backward parabolic regime exists. As expected, the qualitative behavior of $u(t,x)$ is similar to the solution of the heat equation with the same initial datum, namely it has a profile that resembles a hill that continues to decrease until it disappears as $t\to +\infty$.

Now let us approximate $u_{0}(x)$ with a sequence $u_{0n}(x)$ of piecewise constant functions, and let us consider the solution $u_{n}(t,x)$ of the semi-discrete version of (\ref{defn:PM-eqn}) with $u_{0n}(x)$ as initial datum. Then we expect that $u_{n}(t,x)$ mimics $u(t,x)$, but this is true only up to some extent, and it is very sensitive to the choice of the sequence $u_{0n}$. Indeed, it is enough to modify slightly the values of $u_{0n}(x)$ near the maximum point, and what happens is that those values do not follow the evolution of the rest, but they remain rather close to the initial maximum value for large times.

It is still true that $u_{0n}(x)$ converges uniformly to $u_{0}(x)$, and also the maximum and the total variation of $u_{0n}(x)$ converge to the corresponding quantities of $u_{0}(x)$. It is still true that $u_{n}(t,x)$ converges to $u(t,x)$ in $L^{p}((0,\pi))$ for every $t\geq 0$ and every finite $p\geq 1$. But the maximum of $u_{n}(t,x)$ does not tend to the maximum of $u(t,x)$ for every $t>0$, and the same for the total variation, due to the few anomalous values that remain higher than expected.

One could interpret this pathology by saying that $u_{n}(t,x)$ does not converge truly to $u(t,x)$, but to some exotic object that coincides with $u(t,x)$ with an isolated anomalous maximum point, and analogously $u_{0n}(x)$ does not converge truly to $u_{0}(x)$. This approach might be viable, and we plan to explore it in some future research, but it requires the identification of a suitable generalization of the notion of function and derivative, maybe in the spirit of varifolds as in~\cite{GP:fastPM-CdV}. Within this approach the dirty job is done by the definition of the exotic objects, and at the end of the day the convergence of maxima/minima and of the total variation are enforced in the definition, and as a consequence the maximum principle and the monotonicity of the total variation for the limit objects are immediate.

\paragraph{\textmd{\textit{Second main result -- Monotonicity properties of generalized solutions}}}

In this paper we pursue a different path. We remain in the classical setting of functions with bounded variation, where the compactness result provides the limit, and in our second main result (Theorem~\ref{main:monotonicity}) we show that all possible limits satisfy the expected monotonicity properties, even if we know that the related quantities do not pass to the limit from the semi-discrete to the continuous setting.

In order to prove this result, we exploit a weak characterization of generalized solutions. It is known (see Theorem~\ref{thmbibl:v}) that they satisfy an equation of the form $u_{t}=v_{x}$ for a suitable function $v$. From (\ref{defn:PM-eqn}) one would expect that
\begin{equation}
v=\frac{u_{x}}{1+u_{x}^{2}}=\varphi'(u_{x}),
\nonumber
\end{equation}
but in general this is not true, at least if one understands $u_{x}$ as the standard derivative of a function with bounded variation. Nevertheless, it remains true that $v\cdot u_{x}\geq 0$ in the sense of measures (note that for almost every $t\geq 0$ the function $v$ is continuous in $x$, and $u_{x}$ is a signed measure).

This leads us to introducing a class of evolution curves that we call $uv$-evolutions (see Definition~\ref{defn:uv-NBC}), namely solutions to an evolution equation of the form $u_{t}=v_{x}$, with suitable regularity requirements and the sign condition $v\cdot u_{x}\geq 0$. Then we prove that all our generalized solutions of the Perona-Malik equation are actually $uv$-evolutions (see Proposition~\ref{prop:joining-link}), and that all $uv$-evolutions satisfy the maximum principle and the monotonicity of the total variation (see Proposition~\ref{prop:uv-1}).

The class of $uv$-evolutions contains weak solutions to many parabolic equations in divergence form, including all weak solutions to the Perona-Malik equation constructed through convex integration techniques (see~\cite{2006-JDE-Zhang,2015-SIAM-KimYan}, or the more recent~\cite{2018-Nonlinearity-KimYan} and the references quoted therein). The class of $uv$-evolutions is stable under weak convergence, and therefore we are confident that it could contain also the limits of trajectories provided by different approximations of the Perona-Malik equation, of course when a compactness result will be available for those models. In all these cases we have now identified a common mechanism that leads to the maximum principle and to the monotonicity of the total variation, and we hope that this could be useful also in different contexts.

Finally, we conclude by mentioning a technical point in our proofs that might be interesting in itself. Indeed, the maximum principle and the monotonicity of the total variation seem to be two faces of the same coin. More precisely, in order to show that the total variation of $u(t,x)$ with respect to the space variable $x\in(a,b)$ is a nonincreasing function of time, we consider for every positive integer $m$ the function with $2m$ space variables
\begin{equation}
\sum_{i=1}^{2m}(-1)^{i}u(t,x_{i}),
\nonumber
\end{equation}
defined in the simplex $a<x_{1}\leq x_{2}\leq\ldots\leq x_{2m}<b$, and we show that its supremum is a nonincreasing function of time. In other words, the monotonicity of the total variation in dimension one can be reduced to a maximum principle in higher dimension.

\paragraph{\textmd{\textit{Structure of the paper}}}

This paper is organized as follows. In section~\ref{sec:statements} we introduce the notation and we state our main results. In section~\ref{sec:no-strict} we construct our counterexample to the preservation of strict convergence for positive times. In section~\ref{sec:uv} we develop the theory of $uv$-evolutions that leads to the proof of our monotonicity results.

%\clearpage

\setcounter{equation}{0}
\section{Notations and statements}\label{sec:statements}

For the sake of simplicity, throughout this paper we consider a function $\varphi:\re\to\re$ with the following properties, even if many statements could be proved also under weaker assumptions.
\begin{itemize}

\item  (Regularity, symmetry, strict convexity in the origin). The function $\varphi$ is an even function of class $C^{2}$ with bounded second derivative, satisfying $\varphi(0)=0$ and
\begin{equation}
\varphi''(0)>0.
\label{hp:phi''(0)}
\end{equation}

\item  (Convex-concavity). There exists a positive real number $\sigma_{1}$ such that
\begin{equation}
\varphi'(\sigma) \text{ is nondecreasing in $[0,\sigma_{1}]$}
\label{hp:phi'-convex}
\end{equation}
and
\begin{equation}
\varphi'(\sigma) \text{ is nonincreasing in $[\sigma_{1},+\infty)$}.
\label{hp:phi'-concave}
\end{equation}

\item  (Sublinear growth). It turns out that
\begin{equation}
\lim_{\sigma\to +\infty}\varphi'(\sigma)=0.
\label{hp:phi'-lim}
\end{equation}

\end{itemize}

The typical examples are the functions in (\ref{defn:varphi}) and (\ref{defn:varphi-more}). We observe that these properties imply in particular that $\varphi'(0)=0$ and $\varphi(\sigma)>0$ for every $\sigma\neq 0$. In addition, $\sigma_{1}$ is a maximum point for $\varphi'(\sigma)$ and
\begin{equation}
\sigma\varphi'(\sigma)\geq 0
\qquad
\forall\sigma\in\re.
\label{hp:svarphi'}
\end{equation}

In the sequel, $BV((a,b))$ denotes the space of functions with bounded variation in an interval $(a,b)\subseteq\re$. For every function $u$ in this space, $Du$ denotes its distributional derivative, which is a signed measure. Every function $u\in BV((a,b))$ coincides almost everywhere with a function that is (for example) left-continuous. We always identify $u$ with this left-continuous representative.

The \emph{total variation of $u$ in $(a,b)$} is the nonnegative real number
\begin{equation}
\TV(u):=\sup\left\{\sum_{i=1}^{m}|u(x_{i})-u(x_{i-1})|: m\geq 1,\ a<x_{0}\leq x_{1}\leq\ldots\leq x_{m}<b\right\},
\nonumber
\end{equation}
and coincides with the total variation of the measure $Du$. The total variation of $u$ is the sum of the \emph{positive total variation} $\TV^{+}(u)$ of $u$, defined by
\begin{equation}
\TV^{+}(u):=\sup\left\{\sum_{i=1}^{2m}(-1)^{i}u(x_{i}): m\geq 1,\ a<x_{1}\leq x_{2}\leq\ldots\leq x_{2m}<b\right\},
\nonumber
\end{equation}
and the \emph{negative total variation} $\TV^{-}(u)$ of $u$, defined as $\TV^{+}(-u)$. The decomposition $\TV(u)=\TV^{+}(u)+\TV^{-}(u)$ corresponds to the Hahn decomposition of $Du$.

\subsection{The semi-discrete scheme}

Let $n$ be a positive integer, and let $u:\{1,\ldots,n\}\to\re$ be any function. Let us extend $u$ by setting
\begin{equation}
u(0)=u(1)
\qquad\text{and}\qquad
u(n+1)=u(n).
\nonumber
\end{equation}

We call \emph{forward discrete derivative} of $u$ with step $1/n$ the function $D_{n}^{+}u$ defined as
\begin{equation}
[D_{n}^{+}u](i):=\frac{u(i+1)-u(i)}{1/n}
\qquad
\forall i\in\{0,1,\ldots,n\},
\nonumber
\end{equation}
and we call \emph{backward discrete derivative} of $u$ with step $1/n$ the function $D_{n}^{-}u$ defined as
\begin{equation}
[D_{n}^{-}u](i):=\frac{u(i-1)-u(i)}{-1/n}
\qquad
\forall i\in\{1,\ldots,n,n+1\}.
\nonumber
\end{equation}

We observe that $[D_{n}^{-}u](i)=[D_{n}^{+}u](i-1)$ for every admissible value of $i$. In analogy with the continuous setting, we set 
\begin{equation}
\|u\|_{\infty}:=\max\{|u(i)|:i\in\{1,\ldots,n\}\},
\label{defn:D-infty}
\end{equation}
and we call discrete total variation of $u$ the number
\begin{equation}
\TV(u):=\frac{1}{n}\sum_{i=1}^{n}|D_{n}^{+}u(i)|=
\sum_{i=1}^{n-1}|u(i+1)-u(i)|.
\label{defn:D-TV}
\end{equation}

A solution to the \emph{semi-discrete Perona-Malik equation} with step $1/n$ is any function
\begin{equation}
u:[0,+\infty)\times\{1,\ldots,n\}\to\re
\label{defn:u-discr}
\end{equation}
that is differentiable with respect to the first variable, and satisfies
\begin{equation}
u'(t,i)=D_{n}^{-}\left(\varphi'\left(D_{n}^{+}u(t,i)\right)\right)
\qquad
\forall t\geq 0,
\quad
\forall i\in\{1,\ldots,n\},
\label{eqn:DPM}
\end{equation}
with the usual understanding that
\begin{equation}
u(t,0)=u(t,1)
\qquad\text{and}\qquad
u(t,n+1)=u(t,n)
\qquad
\forall t\geq 0.
\label{eqn:DNBC}
\end{equation}

Just to avoid any ambiguity, we point out that
\begin{equation}
D_{n}^{-}\left(\varphi'\left(D_{n}^{+}u(t,i)\right)\right)=
\frac{\varphi'\left(D_{n}^{+}u(t,i)\right)-\varphi'\left(D_{n}^{+}u(t,i-1)\right)}{1/n},
\nonumber
\end{equation}
so that equation (\ref{eqn:DPM}) is the discrete version of (\ref{defn:PM-eqn-phi}), and (\ref{eqn:DNBC}) is the discrete version of the Neumann boundary conditions.

\begin{rmk}[From discrete to continuum setting]\label{rmk:discr2cont}
\begin{em}

One can always identify a function $u:\{1,\ldots,n\}\to\re$ with the piecewise constant function $\hatu:(0,1)\to\re$ defined by
\begin{equation}
\hatu(x):=u(\lceil nx\rceil)
\qquad
\forall x\in(0,1),
\nonumber
\end{equation}
where, for every real number $\alpha$, the symbol $\lceil\alpha\rceil$ denotes the smallest integer greater than or equal to $\alpha$. Under this identification, the forward and backward discrete derivatives $D_{n}^{+}u$ and $D_{n}^{-}u$ are actually the forward and backward difference quotients of $\hatu$ of step $1/n$, the number $\|u\|_{\infty}$ defined by (\ref{defn:D-infty}) is the usual $L^{\infty}$ norm of $\hatu$, and the discrete total variation $\TV(u)$ defined by (\ref{defn:D-TV}) is the total variation of $\hatu$ as a bounded variation function.

In an analogous way, one can always associate to every semi-discrete function $u$ as in (\ref{defn:u-discr}) the function $\hatu:[0,+\infty)\times(0,1)\to\re$ defined by
\begin{equation}
\hatu(t,x):=u(t,\lceil nx\rceil)
\qquad
\forall t\geq 0,
\quad
\forall x\in(0,1).
\label{defn:PCu}
\end{equation}

\end{em}
\end{rmk}

In the sequel, with some abuse of notation, we use the same letter $u$ in order to denote the ``same function'' in three different flavors, namely
\begin{itemize}

\item  the discrete function with domain and codomain as in (\ref{defn:u-discr}),

\item  the piecewise constant function defined by the right-hand side of (\ref{defn:PCu}),

\item  the evolution curve from $[0,+\infty)$ to $L^{p}((0,1))$ that associates to each $t\geq 0$ the function $x\mapsto u(t,\lceil nx\rceil)$, thought as an element of $L^{p}((0,1))$. 

\end{itemize}

We hope that this abuse of notation could result less confusing than using three different symbols. 

\begin{rmk}[Discrete Perona-Malik functional]
\begin{em}

One can define the \emph{discrete Perona-Malik functional} with step $1/n$ as
\begin{equation}
\DPM_{n}(u):=\frac{1}{n}\sum_{i=1}^{n}\varphi\left(D_{n}^{+}u(i)\right).
\label{defn:DPM}
\end{equation}

Under the identification of Remark~\ref{rmk:discr2cont}, one can interpret it as a functional defined in the space of piecewise constant functions in $(0,1)$ with steps of width $1/n$. In this sense it turns out that the semi-discrete Perona-Malik equation (\ref{eqn:DPM}), with the discrete Neumann boundary conditions (\ref{eqn:DNBC}), is the gradient flow of (\ref{defn:DPM}) with respect to the metric of $L^{2}((0,1))$. For more details we refer to~\cite{GG:grad-est}. 

\end{em}
\end{rmk}

%\clearpage

We observe that (\ref{eqn:DPM}) is actually a system of $n$ ordinary differential equations, and therefore existence and uniqueness of solutions follow from standard theories. In the next result we summarize the properties of solutions that we need in the sequel. 

\begin{thmbibl}[Existence, uniqueness, and monotonicity]\label{thmbibl:existence}

For every positive integer $n$, and every function $u_{0}:\{1,\ldots,n\}\to\re$, the following statements hold true.
\begin{enumerate}
\renewcommand{\labelenumi}{(\arabic{enumi})}

\item \emph{(Existence and uniqueness).} There exists a unique global solution to the semi-discrete Perona-Malik equation with initial datum $u_{0}$, namely there exists a unique function $u$ that satisfies (\ref{defn:u-discr}), (\ref{eqn:DPM}), (\ref{eqn:DNBC}), and the initial condition
\begin{equation}
u(0,i)=u_{0}(i)
\qquad
\forall i\in\{1,\ldots,n\}.
\nonumber
\end{equation}

\item \emph{(Monotonicity of max/min and total variation).} The three functions
\begin{equation}
t\mapsto\max_{1\leq i\leq n}u(t,i),
\qquad\quad
t\mapsto-\min_{1\leq i\leq n}u(t,i),
\qquad\quad
t\mapsto\frac{1}{n}\sum_{i=1}^{n}|D_{n}^{+}u(t,i)|
\nonumber
\end{equation}
are nonincreasing.

\item  \emph{($L^{2}$ estimate on the time-derivative).} It turns out that
\begin{equation}
\int_{0}^{+\infty}\left(\frac{1}{n}\sum_{i=1}^{n}|u'(t,i)|^{2}\right)\,dt\leq \DPM_{n}(u_{0}),
\label{th:est-ut-DPM}
\end{equation}
where $\DPM_{n}$ is the discrete Perona-Malik functional defined in (\ref{defn:DPM}).

\item  \emph{(Preservation of subcritical regions).} Let $\sigma_{1}$ be the threshold that appears in (\ref{hp:phi'-convex}) and (\ref{hp:phi'-concave}). If $D_{n}^{+}u_{0}(i)\leq\sigma_{1}$ for some index $i$, then $D_{n}^{+}u(t,i)\leq\sigma_{1}$ for every $t\geq 0$. 

\item  \emph{(Preservation of monotonicity).} If the initial datum $u_{0}$ is nondecreasing, then $u(t,i)$ is nondecreasing with respect to the second variable for every $t\geq 0$.

\end{enumerate}

\end{thmbibl}

The proof of the first four statements of Theorem~\ref{thmbibl:existence} is contained in~\cite[Theorem~2.5]{GG:grad-est}, while statement~(5) follows from the monotonicity of the function
\begin{equation}
t\mapsto\frac{1}{n}\sum_{i=1}^{n}\max\left\{-D_{n}^{+}u(t,i),0\right\},
\nonumber
\end{equation}
whose proof is analogous to the monotonicity of the total variation.

%\clearpage

Solutions to the semi-discrete Perona-Malik equation satisfy the following compactness result. 

\begin{thmbibl}[Compactness for the semi-discrete scheme]\label{thmbibl:compactness}

For every positive integer $n$, let $u_{0n}:\{1,\ldots,n\}\to\re$ be a function, and let $u_{n}:[0,+\infty)\times\{1,\ldots,n\}\to\re$ denote the solution to the semi-discrete Perona-Malik equation (\ref{eqn:DPM}), with discrete Neumann boundary conditions (\ref{eqn:DNBC}), and initial datum $u_{0n}$.

Let us assume that 
\begin{equation}
\sup\left\{\|u_{0n}\|_{\infty}+\TV(u_{0n}):n\geq 1\strut\right\}<+\infty.
\label{hp:bound-data}
\end{equation}

Then the sequence $\{u_{n}\}$ is relatively compact in $C^{0}\left([0,+\infty);L^{2}((0,1))\strut\right)$ with respect to the compact-open topology, namely there exist a function $u:[0,+\infty)\to L^{2}((0,1))$, and an increasing sequence $\{n_{k}\}$ of positive integers, such that 
\begin{equation}
\lim_{k\to +\infty}\sup_{t\in[0,T]}
\|u_{n_{k}}(t,\lceil n_{k}x\rceil)-u(t,x)\|_{L^{2}((0,1))}=0
\label{th:un2u}
\end{equation}
for every real number $T>0$.

\end{thmbibl}

\begin{rmk}
\begin{em}

The proof of Theorem~\ref{thmbibl:compactness}, for which we refer to~\cite[Theorem~2.7]{GG:grad-est} (see also~\cite{2008-JDE-BNPT}), is a simple application of the classical Arzelà-Ascoli theorem. The two main ingredients are the following.
\begin{itemize}

\item  For every fixed $t\geq 0$, the sequence $\{u_{n}(t,x)\}$ is relatively compact in $L^{2}((0,1))$ due to the maximum principle and the bound on the total variation.

\item  The sequence $u_{n}$ is uniformly H\"older continuous with exponent 1/2 as a function from $[0,+\infty)$ to $L^{2}((0,1))$. This is due to estimate (\ref{th:est-ut-DPM}), and the fact that (\ref{hp:bound-data}) implies a uniform bound on $\DPM_{n}(u_{0n})$ because of the sublinear growth of $\varphi$.
\end{itemize} 

We observe also that in the compactness statement the $L^{2}$ space can be replaced by any $L^{p}$ space with finite $p\geq 1$ (but not with $p=+\infty$).
\end{em}
\end{rmk}

%\clearpage

\subsection{Generalized solutions obtained through the SD scheme}

The compactness result of Theorem~\ref{thmbibl:compactness} motivates the following procedure. Given any function $u_{0}\in BV((0,1))$, we approximate it with a sequence $\{u_{0n}\}$ of piecewise constant functions with step $1/n$. For each positive integer $n$, we consider the solution $u_{n}$ to the semi-discrete Perona-Malik equation with initial datum $u_{0n}$. Since $u_{0}$ is a bounded variation function, we can choose the approximating sequence in such a way that (\ref{hp:bound-data}) holds true, and this guarantees that the sequence $\{u_{n}\}$ is relatively compact. All possible limits, when also the approximating sequence is allowed to vary, can be considered as some sort of ``generalized solutions'' to (\ref{defn:PM-eqn-phi}) with initial datum $u_{0}$. 

We observe that this procedure, when applied to the heat equation, or to any other forward parabolic equation, delivers the unique classical solution to the equation with initial datum $u_{0}$. In the case of the Perona-Malik equation we end up with the following notion.

\begin{defn}[Generalized solutions]\label{defn:gen-sol}
\begin{em}

A generalized solution to equation (\ref{defn:PM-eqn-phi}) with homogeneous Neumann boundary conditions in the interval $(0,1)$, obtained through semi-discrete approximation, is any function $u\in C^{0}\left([0,+\infty);L^{2}((0,1))\strut\right)$ for which there exist an increasing sequence $\{n_{k}\}$ of positive integers, and a sequence of functions
\begin{equation}
u_{k}:[0,+\infty)\times\{1,\ldots,n_{k}\}\to\re,
\nonumber
\end{equation}
with the following properties. 
\begin{itemize}

\item  (Uniform bounds on initial data). There exists a real number $M$ such that the initial data, defined by $u_{0k}(i)=u_{k}(0,i)$ for every $i\in\{1,\ldots,n_{k}\}$, satisfy 
\begin{equation}
\|u_{0k}\|_{\infty}+\TV(u_{0k})\leq M
\qquad
\forall k\geq 1.
\label{hp:bound-data-k}
\end{equation}

\item  (Semi-discrete equation and discrete Neumann boundary conditions). For every positive integer $k$, the function $u_{k}$ is differentiable with respect to the first variable, and satisfies the semi-discrete Perona-Malik equation
\begin{equation}
u_{k}'(t,i)=D_{n_{k}}^{-}\left(\varphi'\left(D_{n_{k}}^{+}u_{k}(t,i)\right)\right)
\qquad
\forall t\geq 0,
\quad
\forall i\in\{1,\ldots,n_{k}\},
\nonumber
\end{equation}
with the usual understanding (discrete Neumann boundary conditions) that
\begin{equation}
u_{k}(t,0)=u_{k}(t,1)
\qquad\text{and}\qquad
u_{k}(t,n_{k}+1)=u_{k}(t,n_{k})
\qquad
\forall t\geq 0.
\label{eqn:DNBC-k}
\end{equation}

\item  (Convergence). As $k\to +\infty$ it turns out that
\begin{equation}
u_{k}(t,\lfloor n_{k}x\rfloor)\to u(t,x)
\qquad
\text{in }C^{0}\left([0,+\infty);L^{2}((0,1))\right)
\nonumber
\end{equation}
in the sense that (\ref{th:un2u}) holds true for every real number $T>0$.

\end{itemize}

\end{em}
\end{defn}

\begin{rmk}[Restriction property]
\begin{em}

The notion of solution of Definition~\ref{defn:gen-sol} above is slightly more general than the notion of solution of~\cite[Definition~2.8]{GG:grad-est}. Indeed, in the latter there was the further requirement that initial data converge strictly in $BV((0,1))$, namely that the total variation of the initial data $u_{0k}$ of the approximating solutions converges to the total variation of the initial datum $u(0,x)$ of the limit solution. 

Dropping this extra requirement potentially enlarges the set of generalized solutions, but it has the positive effect that now, if we restrict a solution with initial datum $u_{0}$ to some half-line $[T,+\infty)$, what we get is a solution with initial datum $u(T)$.

\end{em}
\end{rmk}

A partial characterization of generalized solutions is provided by the following result. For a proof, we refer to~\cite[Theorem~2.9]{GG:grad-est}. 

\begin{thmbibl}[Regularity of generalized solution]\label{thmbibl:v}

Let $u$ be a generalized solution to equation (\ref{defn:PM-eqn-phi}) with homogeneous Neumann boundary conditions in the interval $(0,1)$, obtained through semi-discrete approximation in the sense of Definition~\ref{defn:gen-sol} with corresponding sequences $\{n_{k}\}$ and $\{u_{k}\}$.

Then the following statements hold true.

\begin{enumerate}
\renewcommand{\labelenumi}{(\arabic{enumi})}

\item  \emph{($H^{1}$ regularity in time).} The function $u$ admits a weak derivative $u_{t}$ with respect to the variable $t$, and
$$u_{t}\in L^{2}((0,+\infty)\times(0,1)).$$

\item  \emph{($BV$ regularity in space).}  For every $t\geq 0$ the function $x\mapsto u(t,x)$ belongs to $BV((0,1))$ and
\begin{equation}
\|u(t,x)\|_{\infty}+\TV(u(t,x))\leq M
\qquad
\forall t\geq 0,
\nonumber
\end{equation}
where $M$ is the constant that appears in (\ref{hp:bound-data-k}), and both the $L^{\infty}$ norm and the total variation are intended with respect to the space variable.

\item  \emph{(Remnants of the equation).} Let us consider the function $v_{k}$ defined by 
\begin{equation}
v_{k}(t,i):=\varphi'\left(D_{n_{k}}^{+}u_{k}(t,i)\right)
\qquad
\forall t\geq 0,
\quad
\forall i\in\{1,\ldots,n_{k}\}.
\label{defn:vk}
\end{equation} 

Then there exists a measurable function $v\in L^{\infty}((0,+\infty)\times(0,1))$ such that
\begin{equation}
v_{k}(t,\lceil n_{k}x\rceil)\rightharpoonup v(t,x)
\qquad
\text{weakly* in }L^{\infty}((0,+\infty)\times(0,1)).
\nonumber
\end{equation}

Moreover, the function $v$ admits a weak derivative $v_{x}$ with respect to the space variable $x$. This derivative satisfies
\begin{equation}
D_{n_{k}}^{-}v_{k}(t,\lceil n_{k}x\rceil)\rightharpoonup v_{x}(t,x)
\qquad
\text{weakly in }L^{2}((0,+\infty)\times(0,1)),
\nonumber
\end{equation}
and
\begin{equation}
u_{t}=v_{x}
\qquad
\text{as elements of }L^{2}((0,+\infty)\times(0,1)).
\nonumber
\end{equation}

Finally, for almost every $t\geq 0$, the function $x\mapsto v(t,x)$ lies in $H^{1}((0,1))$, and hence it is continuous up to the endpoints and satisfies the boundary conditions $v(t,0)=v(t,1)=0$. 

\end{enumerate}

\end{thmbibl}

%\clearpage

\subsection{Main results}

The $L^{\infty}$ norm and the total variation are lower semicontinuous with respect to convergence in $L^{2}$, or more generally in $L^{p}$ with $p<+\infty$. As a consequence, with the notations of Definition~\ref{defn:gen-sol}, we know that
\begin{equation}
\liminf_{k\to +\infty}\|u_{k}(t,i)\|_{\infty}\geq\|u(t,x)\|_{\infty}
\qquad
\forall t\geq 0,
\label{est:liminf-infty}
\end{equation}
and
\begin{equation}
\liminf_{k\to +\infty}\TV(u_{k}(t,i))\geq \TV(u(t,x))
\qquad
\forall t\geq 0.
\label{est:liminf-TV}
\end{equation}

On the other hand, we know from statement~(2) of Theorem~\ref{thmbibl:existence} that, for every fixed $k\geq 1$, the functions $t\mapsto\|u_{k}(t,i)\|_{\infty}$ and $t\mapsto \TV(u_{k}(t,i))$ are nonincreasing. 

Unfortunately, the inequalities in (\ref{est:liminf-infty}) and (\ref{est:liminf-TV}) are not enough to deduce that the functions $t\mapsto\|u(t,x)\|_{\infty}$ and $t\mapsto \TV(u(t,x))$ are nonincreasing as well. This deduction would be possible if we had equalities instead of inequalities in (\ref{est:liminf-infty}) and (\ref{est:liminf-TV}). This would mean that, for every fixed $t\geq 0$, the convergence of $u_{k}(t,\lceil n_{k}x\rceil)$ to $u(t,x)$ could be improved from convergence in $L^{2}((0,1))$ to strict convergence in $BV((0,1))$ in the sense of~\cite[Definition~3.14]{AFP}.

The first main result of this paper is a counterexample to strict convergence. We show that it may happen that initial data converge strictly, but solutions do not converge strictly for every positive time, and as a consequence we have strict inequality in both (\ref{est:liminf-infty}) and (\ref{est:liminf-TV}) for every $t>0$.

\begin{thm}[Potential lack of strict convergence for positive times]\label{main:no-strict}

There exists a generalized solution $u$ to the Perona-Malik equation (\ref{defn:PM-eqn-phi}) with homogeneous Neumann boundary conditions in the interval $(0,1)$, obtained through semi-discrete approximation in the sense of Definition~\ref{defn:gen-sol} with corresponding sequences $\{n_{k}\}$ and $\{u_{k}\}$, which has the following properties.
\begin{enumerate}
\renewcommand{\labelenumi}{(\arabic{enumi})}

\item  \emph{(Strict converge of initial data).} For $t=0$ we have equality in (\ref{est:liminf-infty}) and (\ref{est:liminf-TV}).

\item \emph{(Lack of strict convergence for positive times).} For every $t>0$ the inequalities in (\ref{est:liminf-infty}) and (\ref{est:liminf-TV}) are strict.

\end{enumerate}

\end{thm}

%\clearpage

Due to Theorem~\ref{main:no-strict}, it is not possible to deduce the monotonicity of the functions $t\mapsto\|u(t,x)\|_{\infty}$ and $t\mapsto \TV(u(t,x))$ from the corresponding monotonicities at discrete level. Nevertheless, the second main result of this paper is that those monotonicities hold true anyway.

\begin{thm}[Monotonicity results for generalized solutions]\label{main:monotonicity}

Let $u$ be a generalized solution to the Perona-Malik equation (\ref{defn:PM-eqn-phi}) with homogeneous Neumann boundary conditions in the interval $(0,1)$, obtained through semi-discrete approximation in the sense of Definition~\ref{defn:gen-sol}.

Then the following monotonicity results hold true.
\begin{enumerate}
\renewcommand{\labelenumi}{(\arabic{enumi})}

\item  \emph{(Maximum principle).} For every $t\geq 0$, let $M^{+}(t)$ and $M^{-}(t)$ denote, respectively, the (essential) supremum and infimum of the function $x\mapsto u(t,x)$.

Then the function $t\mapsto M^{+}(t)$ is nonincreasing, while the function $t\mapsto M^{-}(t)$ is nondecreasing.

\item  \emph{(Monotonicity of the total variation).} For every $t\geq 0$, let $\TV^{\pm}(t)$ denote the positive/negative total variation of the function $x\mapsto u(t,x)$.

Then the functions $t\mapsto \TV^{\pm}(t)$ are nonincreasing with respect to time.

\end{enumerate}

\end{thm}

%\clearpage

\subsection{UV-evolutions}

In the proof of Theorem~\ref{main:monotonicity} we forget that our generalized solutions are limits of solutions to the semi-discrete scheme. We limit ourselves to considering them as functions that satisfy the characterization described in Theorem~\ref{thmbibl:v}, together with a suitable sign condition. This leads to the following notion.

\begin{defn}[$uv$-evolution with NBC in dimension one]\label{defn:uv-NBC}
\begin{em}

A \emph{$uv$-evolution with homogeneous Neumann boundary conditions} in an interval $(a,b)\subseteq\re$ is a pair of measurable functions
\begin{equation}
u:(0,+\infty)\times(a,b)\to\re
\qquad\text{and}\qquad
v:(0,+\infty)\times(a,b)\to\re
\nonumber
\end{equation}
with the following properties. 
\begin{itemize}

\item  (Time regularity). The function $u$ admits a weak derivative $u_{t}$ with respect to time, and
\begin{equation}
u_{t}\in L^{1}((0,T)\times(a,b))
\qquad
\forall T>0.
\label{hp:t-reg-u}
\end{equation}

\item  (Space regularity). For almost every $t>0$ it turns out that
\begin{gather}
\text{the function $x\mapsto u(t,x)$ is in $BV((a,b))$},
\label{hp:s-reg-u}
\\[0.5ex]
\text{the function $x\mapsto v(t,x)$ is in $W^{1,1}((a,b))$}.
\label{hp:s-reg-v}
\end{gather}

\item   (Evolution equation).  The functions $u$ and $v$ satisfy
\begin{equation}
u_{t}(t,x)=v_{x}(t,x)
\qquad
\text{in }(0,+\infty)\times(a,b).
\label{hp:eqn-uv}
\end{equation}

\item  (Neumann boundary conditions). For almost every $t>0$ it turns out that 
\begin{equation}
v(t,a)=v(t,b)=0.
\label{hp:NBC-uv}
\end{equation}

\item   (Sign condition).  For almost every $t>0$ it turns out that
\begin{equation}
v(t,x)\cdot Du(t,x)\geq 0
\qquad
\text{as a measure in $(a,b)$.}
\label{hp:sign-uv}
\end{equation}

\end{itemize}

\end{em}
\end{defn}

\begin{rmk}\label{rmk:uv}
\begin{em}

Let us comment on some regularity issues in Definition~\ref{defn:uv-NBC} above.
\begin{itemize}

\item  (Evolution equation). The evolution equation (\ref{hp:eqn-uv}) can be seen both as an equality between functions in $L^{1}((0,T)\times (a,b))$, and as an equality between functions in $L^{1}((a,b))$ for almost every $t>0$.

\item  (Time regularity and initial datum). The time regularity assumption (\ref{hp:t-reg-u}) implies that $u$, as a function from $(0,+\infty)$ to $L^{1}((a,b))$, is continuous, and actually also absolutely continuous, and hence it can be extended up to $t=0$. In particular, all sections $x\mapsto u(t,x)$ are well defined as functions in $L^{1}((a,b))$ for every $t\geq 0$, including the initial datum at $t=0$.

\item  (Neumann boundary conditions). Due to the space regularity assumption (\ref{hp:s-reg-v}), for almost every $t>0$ the function $x\to v(t,x)$ is continuous up to the boundary, and hence the pointwise values in (\ref{hp:NBC-uv}) make sense.

\item  (Sign condition). For almost every $t>0$, the left-hand side of (\ref{hp:sign-uv}) is a well defined signed measure. Indeed, due to the space regularity assumptions (\ref{hp:s-reg-u}) and (\ref{hp:s-reg-v}), it is the product of the continuous function $x\to v(t,x)$ and the signed measure $Du$, which is the derivative of the $BV$ function $x\to u(t,x)$. 
\end{itemize}

\end{em}
\end{rmk}

%\clearpage

The proof of Theorem~\ref{main:monotonicity} follows from the combination of the following two results, where we show that $uv$-evolutions have some monotonicity properties, and our generalized solutions are $uv$-evolutions. The formal statements are the following.

\begin{prop}[Monotonicity properties of $uv$-evolutions with NBC in dimension one]\label{prop:uv-1}

Let $(u,v)$ be a uv-evolution with Neumann boundary conditions in an interval $(a,b)\subseteq\re$, in the sense of Definition~\ref{defn:uv-NBC}.

Then the following monotonicity results hold true (here we always consider the representative of $u$ that is continuous with values in $L^{1}((a,b))$, see Remark~\ref{rmk:uv}).

\begin{enumerate}
\renewcommand{\labelenumi}{(\arabic{enumi})}

\item  \emph{(Maximum principle).} For every $t\geq 0$, let $M^{+}(t)$ and $M^{-}(t)$ denote, respectively, the (essential) supremum and infimum of the function $x\mapsto u(t,x)$.

Then the function $t\mapsto M^{+}(t)$ is nonincreasing, while the function $t\mapsto M^{-}(t)$ is nondecreasing.

\item  \emph{(Monotonicity of the total variation).} For every $t\geq 0$, let $\TV^{\pm}(t)$ denote the positive/negative total variation of the function $x\mapsto u(t,x)$.

Then the functions $t\mapsto \TV^{\pm}(t)$ are nonincreasing.

\end{enumerate}

\end{prop}

\begin{prop}[Joining link]\label{prop:joining-link}

Let $u$ be a generalized solution to equation (\ref{defn:PM-eqn-phi}) with homogeneous Neumann boundary conditions in the interval $(0,1)$, obtained through semi-discrete approximation in the sense of Definition~\ref{defn:gen-sol}. Let $v$ be the function defined in statement~(3) of Theorem~\ref{thmbibl:v}.

Then the pair $(u,v)$ is a uv-evolution with homogeneous Neumann boundary conditions in $(a,b)$ in the sense of Definition~\ref{defn:uv-NBC}.

\end{prop}

\begin{rmk}[Future perspectives]
\begin{em}

We suspect that other monotonicity results might follow from (or at least be related to) the maximum principle, both in the case of generalized solutions to the Perona-Malik equation, and in the more general framework of $uv$-evolutions. We refer for example to the monotonicity of the total variation of the composition $g(u(t,x))$ of $u$ with a continuous function $g$, or even to the monotonicity of the number of elements of level sets of $u$ (to be suitably defined, since $u$ is just in $BV$). 

The latter in particular would lead to an extension of the classical result by S.~Angenent~\cite{1998-Crelle-Angenent} to forward-backward parabolic equations. We plan to investigate these issues in the future.

\end{em}
\end{rmk}

%\clearpage

\setcounter{equation}{0}
\section{A counterexample to strict convergence}\label{sec:no-strict}

In this section we prove Theorem~\ref{main:no-strict}. Before entering into details, we give an outline of the strategy of the proof. Due to assumption (\ref{hp:phi''(0)}), there exist real numbers $0<\lambda_{0}\leq \Lambda_{0}$ and $\sigma_{0}\in(0,\sigma_{1})$ such that
\begin{equation}
\lambda_{0}(\alpha-\beta)\leq
\varphi'(\alpha)-\varphi'(\beta)\leq
\Lambda_{0}(\alpha-\beta)
\qquad
\forall\: 0\leq\beta\leq\alpha\leq\sigma_{0}.
\label{hp:phi-bilip}
\end{equation}

Now let us consider the function
\begin{equation}
u_{0}(x)=\frac{\sigma_{0}}{2}\sin\left(\frac{\pi}{2}x\right)
\qquad
\forall x\in[0,1],
\label{defn:u0}
\end{equation}
and for every positive integer $n$ let us consider the discrete approximation $u_{0n}$ of $u_{0}$ defined by
\begin{equation}
u_{0n}(i):=
\begin{cases}
\displaystyle \frac{\sigma_{0}}{2}\sin\left(\frac{\pi}{2}\frac{i}{n}\right)  & 
\text{if }i\in\{1,\ldots,m_{n}\}, 
\\[2.5ex]
\displaystyle \frac{\sigma_{0}}{2}%\sin\left(\frac{\pi}{2}\frac{i}{n}\right)
+J_{n}\quad  & 
\text{if }i\in\{m_{n}+1,\ldots,n\},
\end{cases}
\label{hp:un-data}
\end{equation}
where $J_{n}\to 0^{+}$ is a suitable sequence of positive real numbers, and $\{m_{n}\}$ is a suitable sequence of integers such that $0<m_{n}<n$ for every $n\geq 2$, and
\begin{equation}
\lim_{n\to +\infty}\frac{m_{n}}{n}=1.
\label{defn:mn}
\end{equation}

Then the following facts hold true.
\begin{enumerate}
\renewcommand{\labelenumi}{(\arabic{enumi})}

\item  The sequence $\{u_{0n}\}$ converges to $u_{0}$ uniformly in $[0,1]$ (with the usual meaning of Remark~\ref{rmk:discr2cont}), and the total variation of $u_{0n}$ converges to the total variation of $u_{0}$.

This is true because of (\ref{defn:mn}) and the fact that $J_{n}\to 0$. Roughly speaking, this means that the perturbation for $i>m_{n}$ is small, both horizontally and vertically.

\item  For every positive integer $n$, let us consider the solution $u_{n}(t,i)$ of the semi-discrete equation (\ref{eqn:DPM}), with initial datum $u_{0n}(i)$, and the understanding that (the first condition mimics a Dirichlet boundary condition in $x=0$, the second one a Neumann boundary condition in $x=1$)
\begin{equation}
u_{n}(t,0)=0
\qquad\text{and}\qquad
u_{n}(t,n+1)=u_{n}(t,n)
\qquad
\forall t\geq 0.
\label{hp:D-DNBC}
\end{equation}

The function $u_{n}(t,i)$ turns out to be increasing with respect to the second variable for every $t\geq 0$, and in particular
\begin{equation}
\|u_{n}(t,i)\|_{\infty}=
\TV(u_{n}(t,i))=
u_{n}(t,n)
\qquad
\forall t\geq 0.
\label{eqn:TV-un}
\end{equation}

\item  It turns out that
\begin{equation}
\limsup_{n\to +\infty}u_{n}(t,\lceil nx\rceil)\leq
\frac{\sigma_{0}}{2}\sin\left(\frac{\pi}{2}x\right)\cdot\exp(-\lambda_{0}t)
\qquad
\forall t\geq 0,
\quad
\forall x\in(0,1),
\label{th:limsup-un}
\end{equation}
and
\begin{equation}
\liminf_{n\to +\infty}u_{n}(t,n)\geq\frac{\sigma_{0}}{2}
\qquad
\forall t\geq 0.
\label{th:un-n}
\end{equation}

These are the two key points of the proof, and they are established in Proposition~\ref{prop:limit-un} by constructing a suitable sub/supersolution.

\item  The sequence $\{u_{n}\}$ fits into the framework of Theorem~\ref{thmbibl:compactness} (we point out that both Theorem~\ref{thmbibl:existence} and Theorem~\ref{thmbibl:compactness} remain valid also with discrete Dirichlet/Neumann boundary conditions). Any limit point $u(t,x)$ is nondecreasing with respect to $x$ for every $t\geq 0$. Due to (\ref{th:limsup-un}), any limit point satisfies
\begin{equation}
\|u(t,x)\|_{\infty}=
\TV(u(t,x))=
u(t,1)\leq
\frac{\sigma_{0}}{2}\exp(-\lambda_{0}t)
\qquad
\forall t\geq 0.
\nonumber
\end{equation}

On the other hand, from (\ref{eqn:TV-un}) and (\ref{th:un-n}) we deduce that
\begin{equation}
\liminf_{n\to +\infty}\|u_{n}(t,i)\|_{\infty}=
\liminf_{n\to +\infty}\TV(u_{n}(t,i))\geq
\frac{\sigma_{0}}{2},
\nonumber
\end{equation}
from which we conclude that there is strict inequality for every $t>0$ both in (\ref{est:liminf-infty}) and in (\ref{est:liminf-TV}). This proves the conclusions of Theorem~\ref{main:no-strict} for the problem with one Dirichlet and one Neumann boundary condition. If we want an example with Neumann boundary conditions in both endpoints, it is enough to reproduce the phenomenon in $(-1,1)$ by extending $u_{n}(t,i)$ as an odd function for negative values of $i$. 

\end{enumerate}

\begin{rmk}[Characterization of the limit]
\begin{em}

It is possible to show that the whole sequence $u_{n}$ converges to the unique classical solution $u(t,x)$ to equation (\ref{defn:PM-eqn-phi}) with initial datum (\ref{defn:u0}) and boundary conditions (of Dirichlet type in $x=0$ and of Neumann type in $x=1$)
\begin{equation}
u(t,0)=u_{x}(t,1)=0.
\nonumber
\end{equation}

The convergence is in $C^{0}\left([0,+\infty);L^{p}((0,1))\strut\right)$ for every finite $p\geq 1$, but of course not for $p=+\infty$, with the usual meaning of Remark~\ref{rmk:discr2cont}.

This fact can be proved either by constructing more refined subsolutions and supersolutions, or by relying on general results concerning the convergence of gradient-flows, as in the proof of~\cite[Theorem~2.10]{GG:grad-est}.

\end{em}
\end{rmk}

\begin{rmk}[Variants of the counterexample]
\begin{em}

The example described above can be generalized to any interval by translation and/or homothety. Moreover, $u_{0}$ and $u_{0n}$ can be extended by periodicity/reflection in order to obtain an example where the anomalous behavior is not at the boundary, but in a neighborhood of some internal maximum/minimum point, as described in the introduction.

\end{em}
\end{rmk}

The rest of this section is devoted to the proof of (\ref{th:limsup-un}) and (\ref{th:un-n}). The main tool is the following comparison result for solutions to (\ref{eqn:DPM}) (see also~\cite{2011-M3AS-BelNovPao} where a similar idea is exploited). 

%\clearpage

\begin{lemma}[Comparison principle for discrete sub/supersolutions]\label{lemma:comparison}

Let $0<m<n$ be two positive integers, and let $T$ be a positive real number. 

Let $u:[0,T]\times\{1,\ldots,n\}\to\re$ and $v:[0,T]\times\{1,\ldots,n\}\to\re$ be two functions with the following properties, in which $\sigma_{1}$ is the constant that appears in (\ref{hp:phi'-convex}) and (\ref{hp:phi'-concave}).
\begin{enumerate}
\renewcommand{\labelenumi}{(\roman{enumi})}

\item  \emph{(``Space'' monotonicity).} The functions $u$ and $v$ are nondecreasing with respect to the second variable for every $t\in[0,T]$.

\item  \emph{(Solution).} The function $u$ is a solution to the semi-discrete equation (\ref{eqn:DPM}) in the interval $[0,T]$, where $u$ is extended to the ``boundary'' according to the discrete Dirichlet/Neumann boundary conditions
\begin{equation}
u(t,0)=0
\qquad\text{and}\qquad
u(t,n+1)=u(t,n)
\qquad
\forall t\in[0,T].
\label{hp:u-agreement}
\end{equation}

\item  \emph{(Relation between initial data).}  The initial data of $u$ and $v$ satisfy
\begin{equation}
u(0,i)<v(0,i)
\qquad
\forall i\in\{1,\ldots,m\},
\nonumber
\end{equation}
and
\begin{equation}
u(0,i)>v(0,i)
\qquad
\forall i\in\{m+1,\ldots,n\}.
\nonumber
\end{equation}

\item  \emph{(Subcritical condition except in $m$).}  The functions $u$ and $v$ are subcritical for $i\neq m$, in the sense that for every $t\in[0,T]$ their discrete derivatives satisfy
\begin{equation}
D_{n}^{+}u(t,i)\leq \sigma_{1}
\quad\text{and}\quad
D_{n}^{+}v(t,i)\leq \sigma_{1}
\qquad
\forall i\in\{1,\ldots,n\}\setminus\{m\}
\label{hp:subcritical}
\end{equation}

\item  \emph{(Supercritical condition in $m$).} When $i=m$ it turns out that
\begin{equation}
D_{n}^{+}v(t,m))\geq\sigma_{1}
\qquad
\forall t\in[0,T].
\label{hp:v-jump}
\end{equation}

\item  \emph{(Sub/supersolution).}  For every $t\in[0,T]$ the function $v$ is a strict supersolution of equation (\ref{defn:DPM}) for $i\leq m$ in the sense that
\begin{equation}
v'(t,i)>D_{n}^{-}\left(\varphi'(D_{n}^{+}v(t,i)\right)
\qquad
\forall i\in\{1,\ldots,m\},
\nonumber
\end{equation}
and a strict subsolution for $i>m$ in the sense that
\begin{equation}
v'(t,i)<D_{n}^{-}\left(\varphi'(D_{n}^{+}v(t,i)\right)
\qquad
\forall i\in\{m+1,\ldots,n\}.
\nonumber
\end{equation}

Like the function $u$, also the function $v$ is extended to the ``boundary'' according to the discrete Dirichlet/Neumann boundary conditions
\begin{equation}
v(t,0)=0
\qquad\text{and}\qquad
v(t,n+1)=v(t,n)
\qquad
\forall t\in[0,T].
\label{hp:v-agreement}
\end{equation}

\end{enumerate}

Then for every $t\in[0,T]$ it turns out that
\begin{equation}
u(t,i)<v(t,i)
\qquad
\forall i\in\{1,\ldots,m\},
\label{th:u<v}
\end{equation}
and
\begin{equation}
u(t,i)>v(t,i)
\qquad
\forall i\in\{m+1,\ldots,n\}.
\label{th:u>v}
\end{equation}

\end{lemma}

%\clearpage

\begin{proof}

Let us set
\begin{equation}
S:=\sup\{\tau\in[0,T]:\text{(\ref{th:u<v}) and (\ref{th:u>v}) hold true for every $t\in[0,\tau]$}\}.
\nonumber
\end{equation}

We observe that (\ref{th:u<v}) and (\ref{th:u>v}) hold true when $t=0$, and therefore a continuity argument implies that $S>0$. We need to show that $S=T$. So we assume by contradiction that $S\in(0,T)$. Then by the maximality of $S$ there exists $i_{0}\in\{1,\ldots,n\}$ such that
\begin{equation}
v(S,i_{0})=u(S,i_{0}).
\label{th:S=}
\end{equation}

Moreover, since (\ref{th:u<v}) and (\ref{th:u>v}) hold true for every $t\in[0,S)$, passing to the limit we obtain that
\begin{equation}
v(S,i)\geq u(S,i)
\qquad
\forall i\in\{1,\ldots,m\},
\label{th:S>=}
\end{equation}
and
\begin{equation}
v(S,i)\leq u(S,i)
\qquad
\forall i\in\{m+1,\ldots,n\}.
\label{th:S<=}
\end{equation}

Due to the discrete boundary conditions (\ref{hp:u-agreement}) and (\ref{hp:v-agreement}), inequality  (\ref{th:S>=}) is true also for $i=0$, and inequality (\ref{th:S<=}) is true also for $i=n+1$. 

Now let us consider the function $w(t,i):=v(t,i)-u(t,i)$, and let us observe that $w(S,i_{0})=0$ because of (\ref{th:S=}). Now we distinguish two cases.

\subparagraph{\textmd{\textit{Case $i_{0}\in\{1,\ldots,m\}$}}}

We observe that $w(t,i_{0})>0$ for every $t\in[0,S)$ because (\ref{th:u<v}) is true in that interval. It follows that $w'(S,i_{0})\leq 0$, and hence
\begin{equation}
v'(S,i_{0})\leq u'(S,i_{0}).
\label{th:v'<=u'}
\end{equation}

Concerning discrete derivatives, we claim that
\begin{equation}
\varphi'\left(D_{n}^{+}v(S,i_{0}-1)\right)\leq\varphi'\left(D_{n}^{+}u(S,i_{0}-1)\right)
\label{est:i0-1-sub}
\end{equation}
and 
\begin{equation}
\varphi'\left(D_{n}^{+}v(S,i_{0})\right)\geq\varphi'\left(D_{n}^{+}u(S,i_{0})\right).
\label{est:i0-sub}
\end{equation}

Indeed, from (\ref{th:S=}) and from (\ref{th:S>=}) with $i=i_{0}-1$ we find that
\begin{equation}
v(S,i_{0})-v(S,i_{0}-1)=
u(S,i_{0})-v(S,i_{0}-1)\leq
u(S,i_{0})-u(S,i_{0}-1),
\nonumber
\end{equation}
and therefore when we divide by $1/n$ we obtain that
\begin{equation}
D_{n}^{+}v(S,i_{0}-1)\leq D_{n}^{+}u(S,i_{0}-1).
\nonumber
\end{equation}

Now from (\ref{hp:subcritical}) we know that these discrete derivatives lie in the interval $[0,\sigma_{1}]$, and therefore from the monotonicity assumption (\ref{hp:phi'-convex}) we deduce (\ref{est:i0-1-sub}).

As for (\ref{est:i0-sub}), in the case where $i_{0}\in\{1,\ldots,m-1\}$ we can apply (\ref{th:S=}) and (\ref{th:S>=}) with $i=i_{0}+1$. We find that
\begin{equation}
v(S,i_{0}+1)-v(S,i_{0})=
v(S,i_{0}+1)-u(S,i_{0})\geq
u(S,i_{0}+1)-u(S,i_{0}),
\nonumber
\end{equation}
and therefore when we divide by $1/n$ we obtain that
\begin{equation}
D_{n}^{+}v(S,i_{0})\geq D_{n}^{+}u(S,i_{0}),
\nonumber
\end{equation}
from which we deduce (\ref{est:i0-sub}) by exploiting again the monotonicity of $\varphi'$ in $[0,\sigma_{1}]$.

Finally, in the case where $i_{0}=m$ we apply (\ref{th:S=}), and (\ref{th:S<=}) with $i=m+1$. In the usual way we find that
\begin{equation}
D_{n}^{+}v(S,m)\leq D_{n}^{+}u(S,m).
\nonumber
\end{equation}

On the other hand, from (\ref{hp:v-jump}) we know that these two discrete derivatives lie in the region $[\sigma_{1},+\infty)$ where $\varphi'$ is nonincreasing, and therefore the last inequality implies (\ref{est:i0-sub}) in the case $i_{0}=m$.

Now from (\ref{est:i0-sub}) and (\ref{est:i0-1-sub}) we deduce that
\begin{eqnarray*}
D_{n}^{-}\left(\varphi'(D_{n}^{+}v(S,i_{0}))\right) & = &
\frac{\varphi'\left(D_{n}^{+}v(S,i_{0})\right)-\varphi'\left(D_{n}^{+}v(S,i_{0}-1)\right)}{1/n}
\\
& \geq &
\frac{\varphi'\left(D_{n}^{+}u(S,i_{0})\right)-\varphi'\left(D_{n}^{+}u(S,i_{0}-1)\right)}{1/n}
\\
& = &
D_{n}^{-}\left(\varphi'(D_{n}^{+}u(S,i_{0}))\right).
\end{eqnarray*}

Since $u$ is a solution, and $v$ is a supersolution for this value of $i_{0}$, we conclude that
\begin{equation}
v'(S,i_{0})>
D_{n}^{-}\left(\varphi'(D_{n}^{+}v(S,i_{0}))\right)\geq
D_{n}^{-}\left(\varphi'(D_{n}^{+}u(S,i_{0}))\right)=
u'(S,i_{0}),
\nonumber
\end{equation}
which contradicts (\ref{th:v'<=u'}).

\subparagraph{\textmd{\textit{Case $i_{0}\in\{m+1,\ldots,n\}$}}}

In this case we observe that $w(t,i_{0})<0$ for every $t\in[0,S)$ because (\ref{th:u>v}) is true in that interval. It follows that $w'(S,i_{0})\geq 0$, and hence
\begin{equation}
v'(S,i_{0})\geq u'(S,i_{0}).
\label{th:v'>=u'}
\end{equation}

Concerning discrete derivatives, in this case it turns out that
\begin{equation}
\varphi'\left(D_{n}^{+}v(S,i_{0}-1)\right)\geq\varphi'\left(D_{n}^{+}u(S,i_{0}-1)\right)
\nonumber
\end{equation}
and
\begin{equation}
\varphi'\left(D_{n}^{+}v(S,i_{0})\right)\leq\varphi'\left(D_{n}^{+}u(S,i_{0})\right),
\nonumber
\end{equation}
and as a consequence
\begin{equation}
D_{n}^{-}\left(\varphi'(D_{n}^{+}v(S,i_{0}))\right) \leq 
D_{n}^{-}\left(\varphi'(D_{n}^{+}u(S,i_{0}))\right).
\nonumber
\end{equation}

Since $u$ is a solution, and $v$ is a subsolution for this value of $i_{0}$, we conclude that
\begin{equation}
v'(S,i_{0})<
D_{n}^{-}\left(\varphi'(D_{n}^{+}v(S,i_{0}))\right)\leq
D_{n}^{-}\left(\varphi'(D_{n}^{+}u(S,i_{0}))\right)=
u'(S,i_{0}),
\nonumber
\end{equation}
which contradicts (\ref{th:v'>=u'}).
\end{proof}

We are now ready to prove (\ref{th:limsup-un}) and (\ref{th:un-n}). 

%\clearpage

\begin{prop}\label{prop:limit-un}

There exist a sequence of positive real numbers $J_{n}\to 0^{+}$, and a sequence of positive integers $\{m_{n}\}$, with $0<m_{n}<n$ for every $n\geq 2$, such that the following statement holds true. 

The sequence $\{u_{n}\}$ of solutions to the semi-discrete equation (\ref{eqn:DPM}), with the discrete Dirichlet/Neumann boundary conditions (\ref{hp:D-DNBC}), and initial datum defined by (\ref{hp:un-data}), satisfies (\ref{th:limsup-un}) and (\ref{th:un-n}).

\end{prop}

\begin{proof}

Before entering into details, let us describe the general strategy. We introduce the function
\begin{equation}
v_{n}(t,i):=
\begin{cases}
\displaystyle
\frac{\sigma_{0}}{2}\sin\left(\frac{\pi}{2}\frac{i}{n}\right)\exp(-\lambda_{0}t)+A_{n}\frac{i}{n}      & 
\text{if }i\in\{1,\ldots,m_{n}\}, 
\\[2ex]
\displaystyle
\frac{\sigma_{0}}{2}+B_{n}-C_{n}\left(1-\frac{i}{n}\right)^{2}-E_{n}t  \quad & 
\text{if }i\in\{m_{n}+1,\ldots,n\},
\end{cases}
\label{defn:vn}
\end{equation}
where $A_{n}$, $B_{n}$, $C_{n}$ and $E_{n}$ are four sequences of nonnegative real numbers that tend to~0 as $n\to +\infty$, and $m_{n}$ is a sequence of integers that satisfies (\ref{defn:mn}) and $0<m_{n}<n$ for $n$ large enough.

At this point we set $J_{n}:=B_{n}+1/n$ and we claim that, if the sequences $A_{n}$, $B_{n}$, $C_{n}$, $E_{n}$, $m_{n}$ are chosen properly, then for every real number $T>0$ there exists a positive integer $n_{0}$ such that for every $n\geq n_{0}$ and for every $t\in[0,T]$ it turns out that
\begin{equation}
u_{n}(t,i)<v_{n}(t,i)
\qquad
\forall i\in\{1,\ldots,m_{n}\},
\label{est:un<vn}
\end{equation}
and
\begin{equation}
u_{n}(t,i)>v_{n}(t,i)
\qquad
\forall i\in\{m_{n}+1,\ldots,n\}.
\label{est:un>vn}
\end{equation}

These two inequalities are enough to conclude. Indeed, from (\ref{defn:mn}) we deduce that for every $x\in(0,1)$ it turns out that $\lceil nx\rceil\leq m_{n}$ when $n$ is large enough (depending on $x$). At this point from (\ref{defn:vn}) and (\ref{est:un<vn}) with $i=\lceil nx\rceil$ it follows that
\begin{equation}
u_{n}(t,\lceil nx\rceil)\leq
v_{n}(t,\lceil nx\rceil)=
\frac{\sigma_{0}}{2}\sin\left(\frac{\pi}{2}\frac{\lceil nx\rceil}{n}\right)\exp(-\lambda_{0}t)+
A_{n}\frac{\lceil nx\rceil}{n}
\nonumber
\end{equation}
when $n$ is large enough. Since $A_{n}\to 0$, letting $n\to +\infty$ we obtain that (\ref{th:limsup-un}) holds true for every $t\in[0,T]$. Since $T>0$ is arbitrary, that inequality is actually true for every $t\geq 0$.

As for (\ref{th:un-n}), from (\ref{defn:vn}) and (\ref{est:un>vn}) with $i=n$ we obtain that for $n$ large enough it turns out that
\begin{equation}
u_{n}(t,n)\geq
v_{n}(t,n)=
\frac{\sigma_{0}}{2}+B_{n}-E_{n}t
\qquad
\forall t\in[0,T].
\nonumber
\end{equation}

Since $B_{n}\to 0$ and $E_{n}\to 0$, letting $n\to +\infty$ we deduce that the inequality in (\ref{th:un-n}) is true for every $t\in[0,T]$, and we conclude by the arbitrariness of $T$.

The two key estimates (\ref{est:un<vn}) and (\ref{est:un>vn}) follow from Lemma~\ref{lemma:comparison}, applied to the functions $u_{n}$ and $v_{n}$, provided that we choose the sequences $A_{n}$, $B_{n}$, $C_{n}$, $E_{n}$, $m_{n}$ in such a way that the assumptions of Lemma~\ref{lemma:comparison} are satisfied. 

\paragraph{\textmd{\textit{Choice of parameters}}}

Let us consider a function $g:[\sigma_{1},+\infty)\to[0,\sigma_{1}]$ such that 
\begin{equation}
\varphi'(g(\sigma))=\varphi'(\sigma)
\qquad
\forall\sigma\geq\sigma_{1},
\label{eqn:g-main}
\end{equation}

We observe that such a function $g$ exists because from assumption (\ref{hp:phi'-concave}) we know that $0\leq\varphi'(\sigma)\leq\varphi'(\sigma_{1})$ for every $\sigma\geq\sigma_{1}$, and from (\ref{hp:phi'-concave}) we know that $\varphi'$ is nondecreasing and surjective as a function from $[0,\sigma_{1}]$ to $[0,\varphi'(\sigma_{1})]$. We observe also that $g(\sigma)\to 0$ as $\sigma\to +\infty$ due to (\ref{hp:phi'-lim}) and the fact that $\varphi'$ is strictly increasing in a neighborhood of the origin because of (\ref{hp:phi'-convex}).

%From our assumptions on $\varphi$ we deduce that its derivative $\varphi'$ increases from 0 to some positive value $\varphi'(\sigma_{1})$ in the interval $[0,\sigma_{1}]$, and then it decreases from $\varphi'(\sigma_{1})$ to 0 in the half-line $[\sigma_{1},+\infty)$. Therefore, we can consider the function $g:[\sigma_{1},+\infty)\to[0,\sigma_{1}]$ defined by
%\begin{equation}
%g(\sigma):=(\varphi')^{-1}(\varphi'(\sigma))
%\qquad
%\forall\sigma\geq \sigma_{1},
%\nonumber
%\end{equation}
%where $(\varphi')^{-1}$ denotes the inverse of the restriction of $\varphi'$ to the interval $[0,\sigma_{1}]$. We observe that 
%\begin{equation}
%\varphi'(g(\sigma))=\varphi'(\sigma)
%\qquad
%\forall\sigma\geq\sigma_{1},
%\label{eqn:g-main}
%\end{equation}
%and that $g(\sigma)\to 0$ as $\sigma\to +\infty$.

Now we consider the two sequences (the key point is that $h_{n}\to 0$ and $nh_{n}\to +\infty$)
\begin{equation}
h_{n}:=\frac{1}{\sqrt{n}},
\qquad\qquad
\mu_{n}:=\left\lceil n\sqrt{g(nh_{n})}\right\rceil+2,
\nonumber
\end{equation}
and then we set
\begin{equation}
A_{n}:=\varphi'(nh_{n})+\frac{\sigma_{0}\lambda_{0}}{2n},
\qquad\quad
C_{n}:=\frac{ng(nh_{n})}{2\mu_{n}-3},
\qquad\quad
E_{n}:=2\Lambda_{0}C_{n}+\frac{1}{n},
\label{defn:ACD}
\end{equation}
where $\lambda_{0}$ and $\Lambda_{0}$ are the constants that appear in (\ref{hp:phi-bilip}), and finally
\begin{equation}
B_{n}:=A_{n}+C_{n}+h_{n}+\sqrt{E_{n}},
\qquad\qquad
m_{n}:=n-\mu_{n}.
\label{defn:Bm}
\end{equation}

We observe that $m_{n}$ satisfies (\ref{defn:mn}) because $\mu_{n}/n\to 0$, and that the sequences $A_{n}$, $B_{n}$, $C_{n}$, $E_{n}$ tend to 0. Moreover, the sequences $A_{n}$, $B_{n}$ and $E_{n}$ are positive, while $C_{n}$ is just nonnegative (because our assumptions admit that $\varphi'(\sigma)$, and hence also $g(\sigma)$, might vanish when $\sigma$ is large enough). Therefore, for every $T>0$ we can choose a positive integer $n_{0}$ such that the following four inequalities
\begin{equation}
\frac{\pi}{2}\frac{\sigma_{0}}{2}+A_{n}\leq\sigma_{0},
\qquad
2C_{n}\leq\sigma_{0},
\qquad
nh_{n}\geq\sigma_{1},
\qquad
\sqrt{E_{n}}-E_{n}T\geq 0
\label{defn:n0}
\end{equation}
hold true for every $n\geq n_{0}$. In the sequel of the proof we check that all the assumptions of Lemma~\ref{lemma:comparison} are satisfied for every $n\geq n_{0}$.

\paragraph{\textmd{\textit{Solution and space monotonicity}}}

The function $u_{n}(t,i)$ is by definition a solution to the semi-discrete equation (\ref{eqn:DPM}) with the discrete Dirichlet/Neumann boundary conditions (\ref{hp:u-agreement}), and it is nondecreasing with respect to $i$ because of statement~(5) of Theorem~\ref{thmbibl:existence}. 

As for $v_{n}$, from the explicit formula (\ref{defn:vn}) it is immediate that $v_{n}(t,i+1)>v_{n}(t,i)$ at least when $i\neq m_{n}$. In addition, when $i=m_{n}$ we obtain that
\begin{equation}
v_{n}(t,m_{n}+1)\geq
\frac{\sigma_{0}}{2}+B_{n}-C_{n}-E_{n}T
\qquad\text{and}\qquad
v_{n}(t,m_{n})\leq
\frac{\sigma_{0}}{2}+A_{n}.
\nonumber
\end{equation}

Recalling the definition of $B_{n}$ in (\ref{defn:Bm}), and the fourth relation in (\ref{defn:n0}), we conclude that
\begin{equation}
v_{n}(t,m_{n}+1)-v_{n}(t,m_{n})\geq
h_{n}+\left(\sqrt{E_{n}}-E_{n}T\right)\geq
h_{n},
\label{est:vn-vn}
\end{equation}
which proves that the difference is positive also in this case.

\paragraph{\textmd{\textit{Relations between initial data}}}

We need to check that for $n\geq n_{0}$ it turns out that
\begin{equation}
u_{n}(0,i)<v_{n}(0,i)
\qquad
\forall i\in\{1,\ldots,m_{n}\}
\nonumber
\end{equation}
and
\begin{equation}
u_{n}(0,i)>v_{n}(0,i)
\qquad
\forall i\in\{m_{n}+1,\ldots,n\}.
\nonumber
\end{equation}

Since we set $J_{n}:=B_{n}+1/n$, both inequalities are immediate from (\ref{hp:un-data}) and (\ref{defn:vn}).

\paragraph{\textmd{\textit{Subcritical condition except in $m$}}}

We show that for every $t\in[0,T]$ and every $n\geq n_{0}$ it turns out that
\begin{equation}
D_{n}^{+}u_{n}(t,i)\leq\sigma_{1}
\qquad\text{and}\qquad
D_{n}^{+}v_{n}(t,i)\leq\sigma_{0}\leq\sigma_{1}
\label{th:subcritical}
\end{equation}
for every $i\in\{1,\ldots,n\}\setminus\{m_{n}\}$.

As for $u_{n}$, the estimate follows from statement~(4) of Theorem~\ref{thmbibl:existence}, after observing that $|D_{n}^{+}u_{0n}(i)|\leq\sigma_{0}\leq\sigma_{1}$ for all admissible indices $i\neq m_{n}$. 

As for $v_{n}$, we distinguish two cases. When $i\in\{1,\ldots,m_{n}-1\}$, from the explicit formula (\ref{defn:vn}) and the Lipschitz continuity of the function $\sin\sigma$ we obtain that
\begin{equation}
D_{n}^{+}v_{n}(t,i)=
n\frac{\sigma_{0}}{2}
\left\{\sin\left(\frac{\pi}{2}\frac{i+1}{n}\right)-\sin\left(\frac{\pi}{2}\frac{i}{n}\right)\right\}
\exp(-\lambda_{0}t)+A_{n}\leq
n\frac{\sigma_{0}}{2}\frac{\pi}{2n}+A_{n},
\nonumber
\end{equation}
and the latter is less than or equal to $\sigma_{0}$ because of the first condition in (\ref{defn:n0}). 

In the case where $i\in\{m_{n}+1,\ldots,n\}$, from the explicit formula (\ref{defn:vn}) we obtain that
\begin{equation}
D_{n}^{+}v_{n}(t,i) = 
C_{n}\left(2-\frac{2i+1}{n}\right)\leq
2C_{n}\leq
\sigma_{0},
\nonumber
\end{equation}
where in the last inequality we exploited the second condition in (\ref{defn:n0}). 

\paragraph{\textmd{\textit{Supercritical condition in $m$}}}

We show that for every $n\geq n_{0}$ it turns out that
\begin{equation}
D_{n}^{+}v_{n}(t,m_{n})\geq nh_{n}\geq\sigma_{1}
\qquad
\forall t\in[0,T],
\label{est:jump-vn}
\end{equation}
and as a consequence
\begin{equation}
\varphi'(D_{n}^{+}v_{n}(t,m_{n}))\leq
\varphi'(nh_{n})=
\varphi'\left(g(nh_{n})\right)
\qquad
\forall t\in[0,T].
\label{est:jump-sup}
\end{equation}

Indeed, dividing (\ref{est:vn-vn}) by $1/n$, and recalling the third condition in (\ref{defn:n0}), we obtain exactly (\ref{est:jump-vn}). At this point, the two relations in (\ref{est:jump-sup}) follow from (\ref{est:jump-vn}), assumption (\ref{hp:phi'-concave}), and (\ref{eqn:g-main}).

\paragraph{\textmd{\textit{Supersolution for $i\in\{1,\ldots,m_{n}-1\}$}}}

We need to show that for every $n\geq n_{0}$ and every $t\in[0,T]$ it turns out that
\begin{equation}
v_{n}'(t,i)>
D_{n}^{-}\left(\varphi'(D_{n}^{+}v_{n}(t,i))\right)
\qquad
\forall i\in\{1,\ldots,m_{n}-1\}.
\label{th:v-supsol}
\end{equation}

Computing the time derivative, and rearranging the terms, this inequality can be rewritten as
\begin{equation}
n\left\{\varphi'(D_{n}^{+}v_{n}(t,i-1))-\varphi'(D_{n}^{+}v_{n}(t,i))\right\}>
\lambda_{0}\frac{\sigma_{0}}{2}\sin\left(\frac{\pi}{2}\frac{i}{n}\right)\exp(-\lambda_{0}t).
\label{th:v-supsol-equiv}
\end{equation}

From the explicit formula (\ref{defn:vn}), and the concavity of the function $\sin\sigma$, we find that
\begin{equation}
D_{n}^{+}v_{n}(t,i-1)\geq D_{n}^{+}v_{n}(t,i).
\nonumber
\end{equation}

From (\ref{th:subcritical}) we know that both discrete derivatives lie in the interval $[0,\sigma_{0}]$, and hence from the estimate from below in (\ref{hp:phi-bilip}) we deduce that
\begin{equation}
\varphi'(D_{n}^{+}v_{n}(t,i-1))-\varphi'(D_{n}^{+}v_{n}(t,i))\geq
\lambda_{0}\left\{D_{n}^{+}v_{n}(t,i-1)-D_{n}^{+}v_{n}(t,i)\right\}.
\label{est:phi'>=m0}
\end{equation}

Now from the trigonometric identity
\begin{equation}
\sin(a+h)+\sin(a-h)-2\sin(a)=-4\sin(a)\sin^{2}\left(\frac{h}{2}\right),
\nonumber
\end{equation}
applied with $a:=(\pi/2)(i/n)$ and $h:=(\pi/2)(1/n)$, we find that
\begin{eqnarray*}
D_{n}^{+}v_{n}(t,i-1)-D_{n}^{+}v_{n}(t,i) & = &
4\frac{\sigma_{0}}{2}n\sin^{2}\left(\frac{\pi}{2}\frac{1}{2n}\right)\sin\left(\frac{\pi}{2}\frac{i}{n}\right)\exp(-\lambda_{0}t)
\\
& > &
\frac{\sigma_{0}}{2n}\sin\left(\frac{\pi}{2}\frac{i}{n}\right)\exp(-\lambda_{0}t),
\end{eqnarray*}
where in the last step we exploited that $\sin(\frac{\pi}{2}x)> x$ for every $x\in(0,1)$. 

Plugging this inequality into (\ref{est:phi'>=m0}) we obtain (\ref{th:v-supsol-equiv}).

\paragraph{\textmd{\textit{Supersolution for $i=m_{n}$}}}

We need to show that, for every $n\geq n_{0}$ and every $t\in[0,T]$, the  inequality in (\ref{th:v-supsol}) is satisfied also for $i=m_{n}$. 

After computing the derivative and rearranging the terms, the inequality can be rewritten in the form
\begin{equation*}
n\varphi'(D_{n}^{+}v_{n}(t,m_{n}))<
n\varphi'(D_{n}^{+}v_{n}(t,m_{n}-1))-
\frac{\sigma_{0}}{2}\lambda_{0}\sin\left(\frac{\pi}{2}\frac{m_{n}}{n}\right)\exp(-\lambda_{0}t).
\end{equation*}

From the explicit formula (\ref{defn:vn}), and the monotonicity of the function $\sin\sigma$, we obtain that $D_{n}^{+}v_{n}(t,m_{n}-1)>A_{n}$. Therefore, from (\ref{est:jump-sup}) and the definition of $A_{n}$ in (\ref{defn:ACD}) we conclude that
\begin{multline*}
\qquad
n\varphi'(D_{n}^{+}v_{n}(t,m_{n}))\leq
n\varphi'(nh_{n})=
nA_{n}-\frac{\sigma_{0}}{2}\lambda_{0}
\\[0.5ex]
<
n\varphi'(D_{n}^{+}v_{n}(t,m_{n}-1))-
\frac{\sigma_{0}}{2}\lambda_{0}\sin\left(\frac{\pi}{2}\frac{m_{n}}{n}\right)\exp(-\lambda_{0}t),
\qquad
\end{multline*}
as required.

\paragraph{\textmd{\textit{Subsolution for $i=m_{n}+1$}}}

We need to show that for every $n\geq n_{0}$ and every $t\in[0,T]$ it turns out that
\begin{equation}
v_{n}'(t,m_{n}+1)<
\frac{\varphi'(D_{n}^{+}v_{n}(t,m_{n}+1))-\varphi'(D_{n}^{+}v_{n}(t,m_{n}))}{1/n}.
\nonumber
\end{equation}

The left-hand side is equal to $-E_{n}$, and hence negative. Therefore, it is enough to show that the right-hand side is nonnegative. To this end, from the explicit formula (\ref{defn:vn}) we obtain that
\begin{equation}
D_{n}^{+}v_{n}(t,m_{n}+1)=
C_{n}\frac{2(n-m_{n})-3}{n}=
C_{n}\frac{2\mu_{n}-3}{n}.
\nonumber
\end{equation}

At this point, from the definition of $C_{n}$ in (\ref{defn:ACD}), and estimate (\ref{est:jump-sup}), we conclude that
\begin{equation}
\varphi'(D_{n}^{+}v_{n}(t,m_{n}+1))=
\varphi'\left(C_{n}\frac{2\mu_{n}-3}{n}\right)=
\varphi'(g(nh_{n}))\geq
\varphi'(D_{n}^{+}v_{n}(t,m_{n})),
\nonumber
\end{equation}
as required.

\paragraph{\textmd{\textit{Subsolution for $i\in\{m_{n}+2,\ldots,n-1\}$}}}

We need to show that for every $n\geq n_{0}$ and every $t\in[0,T]$ it turns out that
\begin{equation}
v_{n}'(t,i)<
D_{n}^{-}\left(\varphi'(D_{n}^{+}v_{n}(t,i))\right)
\qquad
\forall i\in\{m_{n}+2,\ldots,n-1\}.
\label{th:v-subsol}
\end{equation}

From the explicit formula (\ref{defn:vn}) we obtain that
\begin{equation}
D_{n}^{+}v_{n}(t,i)=C_{n}\left(2-\frac{2i+1}{n}\right)
\quad\text{and}\quad
D_{n}^{+}v_{n}(t,i-1)=C_{n}\left(2-\frac{2i-1}{n}\right).
\nonumber
\end{equation}

Exploiting the estimate from above in (\ref{hp:phi-bilip}), and the definition of $E_{n}$ in (\ref{defn:ACD}), we conclude that
\begin{gather*}
n\left\{\varphi'(D_{n}^{+}v_{n}(t,i-1))-\varphi'(D_{n}^{+}v_{n}(t,i))\right\}\leq
n\Lambda_{0}\left\{D_{n}^{+}v_{n}(t,i-1)-D_{n}^{+}v_{n}(t,i)\right\}
\\[0.5ex]
= 2\Lambda_{0}C_{n}<
E_{n},
\end{gather*}
which is equivalent to (\ref{th:v-subsol}).

\paragraph{\textmd{\textit{Subsolution for $i=n$}}}

It remains to verify that, for every $n\geq n_{0}$ and every $t\in[0,T]$, the  inequality in (\ref{th:v-subsol}) is satisfied also for $i=n$. Due to the discrete Neumann boundary condition, this inequality reduces to
\begin{equation}
-E_{n}<
-n\varphi'\left(D_{n}^{+}v_{n}(t,n-1)\right).
\label{th:subsol-i=n}
\end{equation}

From the explicit formula (\ref{defn:vn}) we deduce that $D_{n}^{+}v_{n}(t,n-1)=C_{n}/n$. Therefore, exploiting again the estimate from above in (\ref{hp:phi-bilip}) (now with $\beta=0$) and the definition of $E_{n}$ in (\ref{defn:ACD}), we conclude that
\begin{equation}
n\varphi'\left(D_{n}^{+}v_{n}(t,n-1)\right)=
n\varphi'\left(\frac{C_{n}}{n}\right)\leq
\Lambda_{0}C_{n}<
E_{n},
\nonumber
\end{equation}
which proves (\ref{th:subsol-i=n}).
\end{proof}

\begin{rmk}[More general setting]
\begin{em}

We observe that the same proof works if we replace the limiting initial datum $u_{0}(x)$ defined by (\ref{defn:u0}) by any smooth function $\hatu_{0}(x)$ such that $0\leq\hatu_{0}(x)\leq u_{0}(x)$ and $0\leq\hatu_{0x}(x)\leq\sigma_{0}$ for every $x\in(0,1)$.

Concerning the nonlinearity $\varphi$, the essential hypotheses are (\ref{hp:phi-bilip}) and (\ref{hp:phi'-lim}). The convex-concave assumption can be avoided by modifying a little the definition of $v_{n}$ and by showing that the discrete derivatives of $u_{n}$ never enter in the region where $\varphi'(\sigma)>\varphi'(\sigma_{0})$. We skip this technical point that only complicates the proof without introducing essentially new ideas.

\end{em}
\end{rmk}

%\clearpage

\setcounter{equation}{0}
\section{Monotonicity results for UV-evolutions}\label{sec:uv}

\subsection{UV-evolutions in any space dimension}

In order to prove Proposition~\ref{prop:uv-1} we need to extend the notion of $uv$-evolution to any space dimension. The extension is almost straightforward, but here we need to consider a combination of Dirichlet and Neumann boundary conditions.

\begin{defn}[$UV$-evolution with DNBC in any dimension]\label{defn:UVW-DNBC}
\begin{em}

Let $d$ be a positive integer, and let $\Omega\subseteq\re^{d}$ be a bounded open set with Lipschitz boundary. 

A \emph{$UV$-evolution with Dirichlet/Neumann boundary conditions} in $\Omega$ is a pair of measurable functions
\begin{equation}
U:(0,+\infty)\times\Omega\to\re
\qquad\text{and}\qquad
V:(0,+\infty)\times\Omega\to\re^{d}
\nonumber
\end{equation}
with the following properties. 
\begin{itemize}

\item  (Time regularity). The function $U$ admits a weak derivative $U_{t}$ with respect to time, and
\begin{equation}
U_{t}\in L^{1}((0,T)\times\Omega)
\qquad
\forall T>0.
\label{hp:t-reg-U}
\end{equation}

\item  (Space regularity). For almost every $t>0$ it turns out that
\begin{gather}
\text{the function $x\mapsto U(t,x)$ is in $BV(\Omega)$},
\label{hp:s-reg-U}
\\
\text{the function $x\mapsto V(t,x)$ is in $W^{1,1}(\Omega;\re^{d})$}.
\label{hp:s-reg-V}
\end{gather}

\item   (Evolution equation).  The functions $U$ and $V$ solve the equation
\begin{equation}
U_{t}(t,x)=\operatorname{div} V(t,x)
\qquad
\text{in }(0,+\infty)\times\Omega.
\label{hp:eqn-UV}
\end{equation}

\item   (Sign condition).  For almost every $t>0$ it turns out that
\begin{equation}
\langle V(t,x),DU(t,x)\rangle\geq 0
\qquad
\text{as a measure in $\Omega$}.
\label{hp:sign-UV}
\end{equation}

\item  (Dirichlet/Neumann boundary conditions). There exists a nonincreasing function $D_{0}:[0,+\infty)\to\re$ such that for almost every $t> 0$ it turns out that, for almost every $x\in\partial\Omega$ (with respect to the $d-1$ dimensional Hausdorff measure), at least one of the following two conditions
\begin{equation}
U(t,x)\leq D_{0}(t),
\qquad\qquad
\langle V(t,x),\nu(x)\rangle\geq 0
\label{hp:DNBC-UV}
\end{equation}
holds true, where $\nu(x)$ denotes the outer normal to $\partial\Omega$ at point~$x$.

\end{itemize}

\end{em}
\end{defn}

\begin{rmk}\label{rmk:UV}
\begin{em}

As we did in dimension one, let us comment on some delicate regularity issues in Definition~\ref{defn:UVW-DNBC} above.

\begin{itemize}

\item  (Evolution equation). As in dimension one, from (\ref{hp:t-reg-U}) and (\ref{hp:s-reg-V}) we know that (\ref{hp:eqn-UV}) can be seen both as an equality between functions in $L^{1}((0,T)\times\Omega)$, and as an equality between functions in $L^{1}(\Omega)$ for almost every $t>0$.

\item  (Time regularity and initial datum). As in dimension one, the time regularity assumption (\ref{hp:t-reg-U}) implies that $U$ is absolutely continuous as a function from $(0,+\infty)$ to $L^{1}(\Omega)$. In particular, all sections $x\mapsto U(t,x)$ are well defined as functions in $L^{1}(\Omega)$ for every $t\geq 0$, including the initial datum at $t=0$.

\item  (Dirichlet/Neumann boundary conditions). From the space regularity assumptions (\ref{hp:s-reg-U}) and (\ref{hp:s-reg-V}) we know that, for almost every $t>0$, the functions $x\mapsto U(t,x)$ and $x\mapsto V(t,x)$ admit a trace on $\partial\Omega$. This implies that the two conditions in (\ref{hp:DNBC-UV}) make sense.

\item  (Sign condition). The left-hand side of (\ref{hp:sign-UV}) is the sum of $d$ terms that are the product of a function in $W^{1,1}(\Omega)$ and a signed measure. In general this product is not well defined, but in this case the measure is the gradient of a function in $BV(\Omega)$, and therefore it is absolutely continuous with respect to the $d-1$ dimensional Hausdorff measure $\mathcal{H}^{d-1}$, and the precise representative of the Sobolev function $V$ (namely the limit as $r\to 0^{+}$ of its average in the ball with radius $r$ centered in $x$) is defined at $\mathcal{H}^{d-1}$ almost every point. Under these assumptions, the product makes sense as a vector measure.

In any case, the $UV$-evolutions that we consider in this paper have the additional property that the function $t\mapsto V(t,x)$ is continuous with respect to $x$ for almost every $t>0$, in which case the definition of the product is less delicate. 

We observe that, if we decompose the vector measure $Du$ into its jump part $D^{J}u$ and its diffuse part $\widetilde{D}u$ (see~\cite[Definition~3.91]{AFP}), then (\ref{hp:sign-UV}) is equivalent to the sign condition on both components, namely
\begin{equation}
\langle V(t,x),\widetilde{D}U(t,x)\rangle\geq 0
\qquad\text{and}\qquad
\langle V(t,x),D^{J}U(t,x)\rangle\geq 0.
\label{hp:sign-UV-dec}
\end{equation} 

\end{itemize}

\end{em}
\end{rmk}

In the next result we show that these $UV$-evolutions in any dimension satisfy a maximum principle.

%\clearpage

\begin{prop}[Maximum principle for $UV$-evolutions with DNBC in any space dimension]\label{prop:UVW}

Let $d$ be a positive integer, and let $\Omega\subseteq\re^{d}$ be a bounded open set with Lipschitz boundary. 

Let $(U,V)$ be a UV-evolution with Dirichlet/Neumann boundary conditions in $\Omega$, in the sense of Definition~\ref{defn:UVW-DNBC}, and let $D_{0}:[0,+\infty)\to\re$ be the nonincreasing function that appears in (\ref{hp:DNBC-UV}).

Then the function defined by 
\begin{equation}
M(t):=\max\left\{D_{0}(t),\operatorname{(ess)sup}\{U(t,x):x\in\Omega\}\strut\right\}
\qquad
\forall t\geq 0
\label{defn:M(t)}
\end{equation}
is nonincreasing (in the definition of $M(t)$ we consider the representative of $U$ that is continuous with values in $L^{1}(\Omega)$, see Remark~\ref{rmk:UV}).

\end{prop}

\begin{proof}

Since we can always restrict to a smaller time interval, it is enough to show that $M(t)\leq M(0)$ for every $ t\geq 0$. On the other hand, from the monotonicity of $D(t)$ it follows that $D_{0}(t)\leq D_{0}(0)\leq M(0)$, and hence we can limit ourselves to showing that for every $t\geq 0$ it turns out that
\begin{equation}
U(x,t)\leq M(0)
\qquad
\text{for almost every $x\in\Omega$.}
\label{th:U<M(0)}
\end{equation}

To this end, we consider a convex function $\psi\in C^{2}(\re)$ that is Lipschitz continuous with 
\begin{equation}
\psi(\sigma)=0
\qquad\text{if and only if}\qquad
\sigma\leq M(0),
\label{hp:psi=0}
\end{equation}
and then we set
\begin{equation}
E(t):=\int_{\Omega}\psi(U(t,x))\,dx
\qquad
\forall t\geq 0.
\label{defn:E(t)}
\end{equation}

We observe that (\ref{hp:psi=0}) and the convexity of $\psi$ imply that $\psi(\sigma)\geq 0$ for every $\sigma\in\re$, and hence
\begin{equation}
E(t)\geq 0
\qquad
\forall t\geq 0,
\label{eqnW:E>=0}
\end{equation}
and in addition $\psi(\sigma)>0$ for every $\sigma>M(0)$, from which we deduce that 
\begin{equation}
\text{$E(t)=0$ if and only if (\ref{th:U<M(0)}) holds true}.
\label{eqnW:E-iff-M}
\end{equation}

Moreover, the function $E(t)$ is absolutely continuous because of the boundedness of $\psi'$ and the time regularity (\ref{hp:t-reg-U}) of $U$. We claim that $E'(t)\leq 0$ for almost every $t\geq 0$. Since $E(0)=0$,  this claim, combined with  (\ref{eqnW:E>=0}), would imply that $E(t)=0$ for every $t\geq 0$, and this would complete the proof because of (\ref{eqnW:E-iff-M}).

Using (\ref{hp:eqn-UV}), we can write the time-derivative of the integral (\ref{defn:E(t)}) in the form
\begin{equation}
E'(t)=
\int_{\Omega}\psi'(U(t,x))U_{t}(t,x)\,dx=
\int_{\Omega}\psi'(U(t,x))\operatorname{div}V(t,x)\,dx.
\nonumber
\end{equation}

Now the space regularity of $U$ and $V$ is enough to integrate by parts, leading to (with some abuse of notation, because the scalar product is a measure and not a function)
\begin{equation}
E'(t)=-\int_{\Omega}\langle D[\psi'(U(t,x))],V(t,x)\rangle\,dx.
\label{eqn:E'(t)}
\end{equation}

In the integration by parts we neglected the boundary term
\begin{equation}
\int_{\partial\Omega}\psi'(U(t,x))
\langle V(t,x),\nu(x)\rangle\,d\sigma
\nonumber
\end{equation}
which is equal to~0 for almost every $t\geq 0$ because of (\ref{hp:DNBC-UV}). Indeed, for almost every $x\in\partial\Omega$ we know that either the scalar product is equal to~0, or
\begin{equation}
U(t,x)\leq
D_{0}(t)\leq
D_{0}(0)\leq
M(0),
\nonumber
\end{equation}
in which case $\psi'(U(t,x))=0$ because of (\ref{hp:psi=0}). Now from the chain rule for bounded variation functions (see~\cite[Theorem~3.96]{AFP}) we know that
\begin{equation}
D[\psi'(U(t,x))]=
\psi''(U(t,x))\widetilde{D}U(t,x)+
\frac{\psi'(U^{+}(t,x))-\psi'(U^{-}(t,x))}{U^{+}(t,x)-U^{-}(t,x)}D^{J}U(t,x),
\nonumber
\end{equation}
where $DU=\widetilde{D}U+D^{J}U$ is the usual decomposition of the vector measure $DU$, and $U^{+}$ and $U^{-}$ are the traces of $U$ on the two sides of the jump set.

Plugging this equality into (\ref{eqn:E'(t)}) we deduce that (again with some abuse of notation, because the two scalar products are measures)
\begin{eqnarray*}
E'(t)& = & -\int_{\Omega}\psi''(U(t,x))\langle\widetilde{D}U(t,x),V(t,v)\rangle\,dx
\\[1ex]
& &
\mbox{}-\int_{\Omega}\frac{\psi'(U^{+}(t,x))-\psi'(U^{-}(t,x))}{U^{+}(t,x)-U^{-}(t,x)}
\langle D^{J}U(t,x),V(t,v)\rangle\,dx,
\end{eqnarray*}
and we conclude by observing that both integrals are nonnegative because of the convexity of $\psi$ and the sign conditions (\ref{hp:sign-UV-dec}).
\end{proof}

%\clearpage

\begin{rmk}[Case with only Neumann boundary conditions]\label{rmk:Neumann-only}
\begin{em}

From the proof it is clear that, when for almost every $t>0$ the second condition in (\ref{hp:DNBC-UV}) is satisfied for almost every $x\in\partial\Omega$, then $D_{0}(t)$ plays no role. In particular, we do not need to consider the maximum with $D_{0}(t)$ in (\ref{defn:M(t)}), or equivalently we can take $D_{0}(t)\equiv -\infty$.

\end{em}
\end{rmk}

\subsection{Proof of Proposition~\ref{prop:uv-1}}

\paragraph{\textmd{\textit{Maximum principle}}}

We claim that the pair
\begin{equation}
U(t,x):=u(t,x),
\qquad\qquad
V(t,x):=v(t,x)
\nonumber
\end{equation}
is a UV-evolution with Dirichlet/Neumann (and actually just Neumann in this case) boundary conditions according to Definition~\ref{defn:UVW-DNBC} with $d=1$, $\Omega=(a,b)$, and no need of $D_{0}(t)$ (see Remark~\ref{rmk:Neumann-only}).
%\begin{equation}
%d:=1,
%\qquad\quad
%\Omega:=(a,b),
%\qquad\quad
%D_{0}(t):=-\|u\|_{L^{\infty}((0,+\infty)\times(a,b))}.
%\nonumber
%\end{equation}

Indeed, all the assumption on $U$ and $V$ in Definition~\ref{defn:UVW-DNBC} follow immediately from the corresponding assumptions on $u$ and $v$ in Definition~\ref{defn:uv-NBC}. 

At this point from Proposition~\ref{prop:UVW} it follows that the function $M(t)$ defined by (\ref{defn:M(t)}) is nonincreasing, but in this case $M(t)$ coincides with the essential supremum $M^{+}(t)$.

The monotonicity of $M^{-}(t)$ can be obtained by applying the maximum principle to the pair $(-u,-v)$, which is again a $uv$-evolution with Neumann boundary conditions. 

\paragraph{\textmd{\textit{Monotonicity of the total variation}}}

To begin with, we observe that it is enough to prove the monotonicity of the positive total variation, because the negative total variation of $u$ is the positive total variation of $-u$, and we have already observed that $(-u,-v)$ is again a $uv$-evolution with Neumann boundary conditions. 

To this end, for every positive integer $m$ we introduce the positive $m$-variation
\begin{equation}
TV_{m}^{+}(t):=\sup\left\{
\sum_{i=1}^{2m}(-1)^{i}u(t,x_{i}):
a< x_{1}\leq x_{2}\leq\ldots\leq x_{2m}< b.
\right\}.
\nonumber
\end{equation}

We observe that
\begin{equation}
\TV^{+}(t)=
\sup_{m\geq 1}TV_{m}^{+}(t)=
\lim_{m\to +\infty}TV_{m}^{+}(t)
\nonumber
\end{equation}

Therefore, if we prove that the function $t\mapsto TV_{m}^{+}(t)$ is nonincreasing for every $m\geq 1$, then thesis follows. We prove the monotonicity of $TV_{m}^{+}(t)$ by induction on $m$.

\subparagraph{\textmd{\textit{Case $m=1$}}}

We claim that the pair defined by
\begin{equation}
U(t,x_{1},x_{2}):=u(t,x_{2})-u(t,x_{1}),
\qquad
V(t,x_{1},x_{2}):=\left(-v(t,x_{1}),v(t,x_{2})\right)
\nonumber
\end{equation}
is a UV-evolution with Dirichlet/Neumann boundary conditions in the sense of Definition~\ref{defn:UVW-DNBC} with 
\begin{equation}
d:=2,
\qquad
\Omega:=\left\{(x_{1},x_{2})\in(a,b)^{2}:a< x_{1}< x_{2}< b\right\},
\qquad
D_{0}(t)\equiv 0.
\nonumber
\end{equation}

If the claim is true, then the monotonicity of $TV_{1}^{+}(t)$ follows from Proposition~\ref{prop:UVW}, because in this case the function $M(t)$ defined by (\ref{defn:M(t)}) coincides with $TV_{1}^{+}(t)$.

So let us check that $U$ and $V$ satisfy the properties in Definition~\ref{defn:UVW-DNBC}. The regularity and the evolution equation follow from the corresponding properties of $u$ and $v$ in Definition~\ref{defn:UVW-DNBC}. The sign condition (\ref{hp:sign-UV}) follows from (\ref{hp:sign-uv}) because
\begin{equation}
\langle V(t,x_{1},x_{2}),DU(t,x_{1},x_{2})\rangle=
v(t,x_{1})\cdot Du(t,x_{1})+v(t,x_{2})\cdot Du(t,x_{2})
\nonumber
\end{equation}
is the sum of two nonnegative measures. 

Finally, we observe that $\Omega$ is a triangle, and its boundary is contained in the three lines described by the three equalities $a=x_{1}$, $x_{1}=x_{2}$, and $x_{2}=b$.
\begin{itemize}

\item  In the side with $a=x_{1}$ the normal vector is $\nu(x_{1},x_{2})=(-1,0)$, and hence 
\begin{equation}
\langle V(t,x_{1},x_{2}),\nu(x_{1},x_{2})\rangle=
v(t,x_{1})=
v(t,a)=
0.
\nonumber
\end{equation}

\item  In the side with $x_{2}=b$ the normal vector is $\nu(x_{1},x_{2})=(0,1)$, and hence 
\begin{equation}
\langle V(t,x_{1},x_{2}),\nu(x_{1},x_{2})\rangle=
v(t,x_{2})=
v(t,b)=
0.
\nonumber
\end{equation}

\item  In the side with $x_{1}=x_{2}$ it turns out that
\begin{equation}
U(t,x_{1},x_{2})=0\leq D_{0}(t).
\nonumber
\end{equation}

\end{itemize}

Therefore, in all the sides of $\partial\Omega$ the Dirichlet/Neumann boundary conditions (\ref{hp:DNBC-UV}) are satisfied, and this completes the proof.

\subparagraph{\textmd{\textit{Inductive step}}}

We assume that $TV_{m}^{+}(t)$ is nonincreasing for some positive integer $m$, and we prove that also $TV_{m+1}^{+}(t)$ is nonincreasing. 

To this end, we consider the pair defined by
\begin{equation}
U(t,x):=\sum_{i=1}^{2m+2}(-1)^{i}u(t,x_{i}),
\qquad\quad
V(t,x):=\sum_{i=1}^{2m+2}(-1)^{i}v(t,x_{i})e_{i},
\nonumber
\end{equation}
where $x=(x_{1},\ldots,x_{2m+2})$, and $e_{i}$ denotes the $i$-th vector of the canonical basis of $\re^{2m+2}$. We claim that this pair is a UV-evolution with Dirichlet/Neumann boundary conditions according to Definition~\ref{defn:UVW-DNBC} with
\begin{gather*}
d:=2m+2,
\qquad\qquad
D_{0}(t):=TV_{m}^{+}(t),
\\[0.5ex]
\Omega:=\left\{(x_{1},x_{2},\ldots,x_{2m+2})\in(a,b)^{2m+2}:a< x_{1}<\ldots< x_{2m+2}< b\right\}.
\end{gather*}

If this is the case, then the monotonicity of $TV_{m+1}^{+}(t)$ follows from Proposition~\ref{prop:UVW}, because \begin{equation}
TV_{m+1}^{+}(t)=\sup\{U(t,x_{1},\ldots,x_{2m+2}):(x_{1},\ldots,x_{2m+2})\in\Omega\},
\nonumber
\end{equation}
and in particular the function $M(t)$ defined by (\ref{defn:M(t)}) in this case is exactly
\begin{equation}
M(t)=
\max\{D_{0}(t),TV_{m+1}^{+}(t)\}=
\max\{TV_{m}^{+}(t),TV_{m+1}^{+}(t)\}=
TV_{m+1}^{+}(t).
\nonumber
\end{equation}

So let us check that $U$ and $V$ satisfy the assumptions in Definition~\ref{defn:UVW-DNBC}. As before, the regularity and the evolution equation follow from the corresponding properties of $u$ and $v$ in Definition~\ref{defn:UVW-DNBC}. The sign condition (\ref{hp:sign-UV}) follows from (\ref{hp:sign-uv}) because
\begin{equation}
\langle V(t,x),DU(t,x)\rangle=
\sum_{i=1}^{2m+2}v(t,x_{i})\cdot Du(t,x_{i})
\nonumber
\end{equation}
is the sum of $2m+2$ nonnegative measures.

Finally, we consider the boundary of $\Omega$, which consists of $2m+3$ ``faces'' contained in the hyperplanes corresponding to the possible equalities in the definition of $\Omega$. 
\begin{itemize}

\item  In the face with $x_{1}=a$ the normal vector is $\nu(x)=-e_{1}$, and hence 
\begin{equation}
\langle V(t,x),\nu(x)\rangle=
v(t,x_{1})=
v(t,a)=
0.
\nonumber
\end{equation}

\item  In the face with $x_{2m+2}=b$ the normal vector is $\nu(x)=e_{2m+2}$, and hence 
\begin{equation}
\langle V(t,x),\nu(x)\rangle=
v(t,x_{2m+2})=
v(t,b)=
0.
\nonumber
\end{equation}

\item  Let us finally consider the faces where $x_{i}=x_{i+1}$ for some index $i$. In this case two consecutive terms in the definition of $U$ cancel, and what remains is a competitor in the definition of $TV_{m}^{+}(t)$. It follows that in all these $2m+1$ faces of $\partial\Omega$ it turns out that
\begin{equation}
U(t,x_{1},\ldots,x_{2m+2})\leq
TV_{m}^{+}(t)=
D_{0}(t).
\nonumber
\end{equation}

\end{itemize}

This proves that the Dirichlet/Neumann boundary conditions (\ref{hp:DNBC-UV}) are satisfied, and thus completes the proof.
\qed

%\clearpage

\subsection{Proof of Proposition~\ref{prop:joining-link}}

The time and space regularity of $u$ and $v$, as well as the evolution equation that they solve and the boundary conditions, follow exactly from Theorem~\ref{thmbibl:v}. It remains to prove that $u$ and $v$ satisfy the sign condition (\ref{hp:sign-uv}) for almost every $t\geq 0$.

To this end, let us consider any test function $\phi\in C^{1}([0,1])$, and its discrete sampling
\begin{equation}
\phi_{k}(i):=\phi\left(\frac{i}{n_{k}}\right)
\qquad
\forall i\in\{0,1,\ldots,n_{k}\}.
\nonumber
\end{equation}

From the space regularity of $u$ and $v$ we know that, for almost every $t\geq 0$, the function $x\mapsto u(t,x)$ lies in $BV((0,1))$, while the function $t\mapsto v(t,x)\phi(x)$ lies in $W^{1,1}((0,1))$ and vanishes at the boundary. Therefore, for any such $t$ it turns out that
\begin{equation}
\int_{0}^{1}Du(t,x)v(t,x)\phi(x)\,dx=
-\int_{0}^{1}u(t,x)v_{x}(t,x)\phi(x)\,dx
-\int_{0}^{1}u(t,x)v(t,x)\phi_{x}(x)\,dx.
\nonumber
\end{equation}

We point out that, as usual, there is a little abuse of notation in the left-hand side because $Du$ is actually a measure. Integrating with respect to time we deduce that
\begin{eqnarray}
\int_{t_{1}}^{t_{2}}dt\int_{0}^{1}Du(t,x)v(t,x)\phi(x)\,dx & = & 
-\int_{t_{1}}^{t_{2}}dt\int_{0}^{1}u(t,x)v_{x}(t,x)\phi(x)\,dx
\nonumber
\\
& &
\mbox{}-\int_{t_{1}}^{t_{2}}dt\int_{0}^{1}u(t,x)v(t,x)\phi_{x}(x)\,dx
\label{eqn:int-ux-v}
\end{eqnarray}
for every choice of the times $t_{2}\geq t_{1}\geq 0$.

Analogously, by a discrete integration by parts (which is actually an algebraic manipulation of the sums, where we exploit also the discrete Neumann boundary conditions (\ref{eqn:DNBC-k})), we obtain that
\begin{multline}
\quad
\sum_{i=1}^{n_{k}}D_{n_{k}}^{+}u_{k}(t,i)\cdot v_{k}(t,i)\cdot\phi_{k}(i)
\\
=-\sum_{i=1}^{n_{k}}u_{k}(t,i)\cdot D_{n_{k}}^{-}v_{k}(t,i)\cdot\phi_{k}(i-1)
-\sum_{i=1}^{n_{k}}u_{k}(t,i)\cdot v_{k}(t,i)\cdot D_{n_{k}}^{-}\phi_{k}(i)
\quad
\label{eqn:ibp-discrete}
\end{multline}
for every $t\geq 0$ and every positive integer $k$. If we integrate with respect to time, and we rewrite the sums as integrals of piecewise constant functions, we deduce that
\begin{multline*}
\quad
\int_{t_{1}}^{t_{2}}dt\int_{0}^{1}
D_{n_{k}}^{+}u_{k}(t,\lceil n_{k}x\rceil)\cdot v_{k}(t,\lceil n_{k}x\rceil)\cdot\phi_{k}(\lceil n_{k}x\rceil)\,dx
\\=
-\int_{t_{1}}^{t_{2}}dt\int_{0}^{1}
u_{k}(t,\lceil n_{k}x\rceil)\cdot D_{n_{k}}^{-}v_{k}(t,\lceil n_{k}x\rceil)\cdot\phi_{k}(\lceil n_{k}x\rceil-1)\,dx
\\
-\int_{t_{1}}^{t_{2}}dt\int_{0}^{1}
u_{k}(t,\lceil n_{k}x\rceil)\cdot v_{k}(t,\lceil n_{k}x\rceil)\cdot D_{n_{k}}^{-}\phi_{k}(\lceil n_{k}x\rceil)\,dx
\qquad
\end{multline*}
for every positive integer $k$ and every choice of the times $t_{2}\geq t_{1}\geq 0$.

Now we are allowed to pass to the limit in the two double integrals of the right-hand side, because in each of them the integrand is the product of two terms that converge strongly (the ones with $u_{k}$ and $\phi_{k}$) and one term that converges weakly in the pair $(t,x)$ (the one with $v_{k}$). Since the limits of these two integrals are the two integrals in the right-hand side of (\ref{eqn:int-ux-v}), we conclude that
\begin{multline}
\lim_{n\to +\infty}\int_{t_{1}}^{t_{2}}dt\int_{0}^{1}
D_{n_{k}}^{+}u_{k}(t,\lceil n_{k}x\rceil)\cdot v_{k}(t,\lceil n_{k}x\rceil)\cdot\phi_{k}(\lceil n_{k}x\rceil)\,dx
\\
=\int_{t_{1}}^{t_{2}}dt\int_{0}^{1}Du(t,x)\cdot v(t,x)\cdot\phi(x)\,dx.
\label{th:ux-v}
\end{multline}

Now we observe that, if the test function $\phi$ is nonnegative, then from (\ref{defn:vk}) and (\ref{hp:svarphi'}) we deduce that in the left-hand side we have a limit of integrals of nonnegative functions, and hence 
\begin{equation}
\int_{t_{1}}^{t_{2}}dt\int_{0}^{1}Du(t,x)\cdot v(t,x)\cdot\phi(x)\,dx\geq 0
\nonumber
\end{equation}
for every choice of the nonnegative test function $\phi$ and of the times $t_{2}\geq t_{1}\geq 0$. Since the times are arbitrary, we deduce that, for every nonnegative test function $\phi$, there exists a subset $E_{\phi}\subseteq[0,+\infty)$ with Lebesgue measure equal to~0 such that
\begin{equation}
\int_{0}^{1}Du(t,x)\cdot v(t,x)\cdot\phi(x)\,dx\geq 0
\qquad
\forall t\in[0,+\infty)\setminus E_{\phi}.
\label{est:key-point}
\end{equation}

The set $E_{\phi}$ might depend on $\phi$, but we can always take a countable set $\mathcal{D}$ of nonnegative test functions that is dense in the nonnegative functions of $C^{1}([0,1])$, and a common subset $E\subseteq[0,+\infty)$ with Lebesgue measure equal to~0, such that
\begin{equation}
\int_{0}^{1}Du(t,x)\cdot v(t,x)\cdot\phi(x)\,dx\geq 0
\qquad
\forall\phi\in\mathcal{D},
\quad
\forall t\in[0,+\infty)\setminus E,
\nonumber
\end{equation}
which guarantees that the sign condition (\ref{hp:sign-uv}) is satisfied for any such $t$.  
\qed

\begin{rmk}
\begin{em}

We observe that (\ref{th:ux-v})  is equivalent to saying that 
\begin{equation}
D_{n_{k}}^{+}u_{k}\cdot v_{k}\rightharpoonup Du\cdot v
\qquad\text{weakly* as measures in }(0,+\infty)\times(0,1).
\nonumber
\end{equation}

We observe also that the key point in the proof is (\ref{est:key-point}). One might be tempted to establish this relation by passing to the limit in the right-hand side of (\ref{eqn:ibp-discrete}) before integrating with respect to time. Indeed, for almost every $t>0$ we have a bound in $L^{\infty}((0,1))$ on the function $x\mapsto v_{k}(t,\lceil n_{k}x\rceil)$, and a bound in $L^{2}((0,1))$ on its discrete derivative, and therefore these functions admit a weak limit up to subsequences.

The problem with this approach is that in Theorem~\ref{thmbibl:v} the function $v(t,x)$ is defined as the weak limit of $v_{k}$ in the pair $(t,x)$, and therefore there is no guarantee that the weak limits of the sections of $v_{k}$ at fixed times have anything to do with the sections of the limit $v$. This forces us to pass through the double integrals.

\end{em}
\end{rmk}

%\clearpage

\subsubsection*{\centering Acknowledgments}

The authos are members of the Italian \selectlanguage{italian} ``Gruppo Nazionale per l'Analisi Matematica, la Probabilità e le loro Applicazioni'' (GNAMPA) of the ``Istituto Nazionale di Alta Matematica'' (INdAM). 

\selectlanguage{english}

%\addcontentsline{toc}{chapter}{Bibliography}

%\bibliographystyle{MaxNew}
%\bibliography{../../../BibTeX/PeronaMalik}

\label{NumeroPagine}

\end{document}